\documentclass[preprint,12pt]{elsarticle}
\def\arx{1} %set to 1 for arxiv version
\usepackage[T1]{fontenc}
\usepackage[utf8]{inputenc}
\usepackage{fullpage}
\usepackage{amsmath}
\usepackage{ifthen}
\usepackage{amsfonts,amsthm, amssymb}
\usepackage{tikz}
\usetikzlibrary{decorations.pathreplacing}
\usepackage{ifthen}

\usepackage{cancel}
\usepackage{mathrsfs}

\usepackage{hyperref}
\usepackage[capitalise,nameinlink, noabbrev]{cleveref}

\usepackage{stmaryrd}
\usepackage{fontawesome}

\colorlet{myGreen}{green!50!black}
\definecolor{myBlue}{rgb}{0.25, 0.0, 1.0}
\definecolor{lgray}{rgb}{0.75, 0.75, 0.75}
\newcommand{\ourtitle}{Sprague-Grundy values and complexity for LCTR}

\hypersetup{
  colorlinks=true,
  linkcolor=myBlue,
  citecolor=myGreen,
  urlcolor=myGreen,
  bookmarksopen=true,
  bookmarksnumbered,
  bookmarksopenlevel=2,
  bookmarksdepth=3
  pdftitle = {\ourtitle},
  pdfauthor= {Eric Gottlieb, Matjaz Krnc, Peter Mursic},
}

\newtheorem{theorem}{Theorem}
\newtheorem{definition}[theorem]{Definition}
\newtheorem{lemma}[theorem]{Lemma}
\newtheorem{proposition}[theorem]{Proposition}
\newtheorem{corollary}[theorem]{Corollary}
\newtheorem{conjecture}[theorem]{Conjecture}

\newtheorem{observation}[theorem]{Observation}

\newcommand{\mexsymb}{\triangledown}
\newcommand{\mex}[2]{#1 \, \mexsymb \, #2}
\newcommand{\mexx}[2]{\left(#1 \, \mexsymb \, #2\right)}
\newcommand{\B}{A}
\renewcommand\L{\mathcal{L}}
\newcommand{\K}{\mathcal{D}}
\DeclareMathOperator{\Mex}{mex}
\newcommand{\D}{d}
\newcommand{\cor}{\textrm{cor}}
\newcommand{\norm}[1]{\textnormal{trunc}(#1)}
\newcommand{\trunc}[1]{\textnormal{trunc}(#1)}
\newcommand{\br}[1]{\llbracket #1 \rrbracket}

\def\P{\mathcal{P}}
\def\N{\mathcal{N}}
\def\G{\mathbb{SG}}
\def\SG{\G}

\def\Young{Young diagram}
\def\Youngs{Young diagrams}
\def\King{Downright}
\newcommand{\nodes}{%
\begin{tikzpicture}[scale=0.10]%
\node[circle,thick,draw,inner sep=1pt] (b) at (0,0) {};
\draw[white] (-0.2,1.6)--(0.8,1.4);
\draw[thick] (-1,-1)--(b)--(1,1);
\draw[thick] (b) -- (1,-1);

\end{tikzpicture}%
}

\newcommand{\leaf}{\text{\scriptsize \faLeaf}}

\newcommand{\gama}{%
\begin{tikzpicture}[scale=0.08]%
\draw (0,0) -- (0,-3)--(1,-3)--(1,0)
    (0,0)--(3,0)--(3,-1)--(0,-1)
    (0,-2)--(1,-2)
    (2,0)--(2,-1);%
\end{tikzpicture}%
}

\newcommand{\staircase}{%
\begin{tikzpicture}[scale=0.08]%
\draw (0,0) -- (0,-3)--(1,-3)--(1,0)
    (0,0)--(3,0)--(3,-1)--(0,-1)
    (0,-2)--(2,-2)
    (2,0)--(2,-2);%
\end{tikzpicture}%
}

\newcommand{\rectangle}{%
\begin{tikzpicture}[scale=0.08]%
\draw (0,0) -- (0,-3)--(1,-3)--(1,0)
    (0,0)--(3,0)--(3,-1)--(0,-1)
    (0,-2)--(3,-2)--(3,-3)--(0,-3)
    (2,0)--(2,-3)--(3,-3)--(3,0);%
\end{tikzpicture}%
}

\newcommand{\picturemex}{
\begin{tikzpicture}
       \node (a) at (6,0.0)
         {
\begin{tikzpicture}[scale=0.5, every node/.style={scale=1}]
\def\x{7,6,5,5,5,3,2,1}
\def\tabx{{2,5}}
\def\taby{{5,3}}
\foreach \i [count=\xi] in \x{
    \draw [step=1.0,blue, very thick] (\xi-1,0) grid (\xi,-\i);
}
\end{tikzpicture}   
         };
        \node (b) at (12,0.25) %[anchor=north,yshift=-1cm]
         {
\begin{tikzpicture}[,scale=0.5, every node/.style={scale=1}]
\def\x{6,5,5,5,3,2,1}
\def\tabx{{2,5}}
\def\taby{{5,3}}
\foreach \i [count=\xi] in \x{
    \draw [step=1.0,blue, very thick] (\xi-1,0) grid (\xi,-\i);
}

\end{tikzpicture}   
         };
                 \node (c) at (5.75,-4.5) %[anchor=north,yshift=-1cm]
         {
\begin{tikzpicture}[scale=0.5, every node/.style={scale=1}]
\def\x{6,5,4,4,4,2,1}
\def\tabx{{2,5}}
\def\taby{{5,3}}
\foreach \i [count=\xi] in \x{
    \draw [step=1.0,blue, very thick] (\xi-1,0) grid (\xi,-\i);
}

\end{tikzpicture}   
         };
        \draw [->,ultra thick] (7.5,0)--(10,0) node
        [midway,above,xshift=-0mm] {$\lambda \to \lambda[0,1] $};
        \draw [->,ultra thick] (5.5,-1.25)--(5.5,-2.75) node [midway,xshift=12mm] {$\lambda \to  \lambda[1,0]$};
    \end{tikzpicture}
}

\newcommand{\gb}[2]{
    \pgfmathsetmacro{\resultone}{0.5+#1}
    \pgfmathsetmacro{\resulttwo}{-0.5-#2}
    \node[fill=lgray,minimum height = 0.6cm,minimum width = 0.6cm] at (\resultone,\resulttwo) {};
}

\newcommand{\tc}[3]{
    \pgfmathsetmacro{\resultone}{0.5+#1}
    \pgfmathsetmacro{\resulttwo}{-0.5-#2}
    \draw (\resultone,\resulttwo) node {#3};
}

\newcommand{\picturepropagate}{
\begin{tikzpicture}
\begin{scope}[scale=0.5, every node/.style={scale=1}]
\def\x{7,6,5,5,5,3,2,1}
\def\tabx{{2,5}}
\def\taby{{5,3}}
\foreach \i [count=\xi] in \x{
    \draw [step=1.0,blue, very thick] (\xi-1,0) grid (\xi,-\i);
}

\foreach \i  in {1,2}{
    \pgfmathsetmacro{\result}{min(\tabx[\i-1],\taby[\i-1])}
   \foreach \j  in {1,2,...,\result}{
        \draw (\tabx[\i-1]+0.5-\j,\j-0.5-\taby[\i-1]) node {0};
   }
\tc000
\tc012
\tc021
\tc030
\tc041
\tc050
\tc061
\tc101
\tc110
\tc122
\tc131
\tc140
\tc151
\tc200
\tc211
\tc220
\tc232
\tc241
\tc301
\tc310
\tc321
\tc330
\tc342
\tc400
\tc411
\tc420
\tc432
\tc441
\tc501
\tc510
\tc521
\tc600
\tc611
\tc701
\tc800
\tc710
\tc620
\tc530
\tc540
\tc450
\tc350
\tc250
\tc160
\tc070
}
\end{scope}
\begin{scope}[scale=0.5, every node/.style={scale=1},xshift=16cm]
\def\x{7,6,5,5,5,3,2,1}
\def\tabx{{2,5}}
\def\taby{{5,3}}
\foreach \i [count=\xi] in \x{
    \draw [step=1.0,blue, very thick] (\xi-1,0) grid (\xi,-\i);
 \tc700
 \tc610
 \tc500
 \tc520
 \tc410
 \tc300
 \tc440
 \tc330
 \tc220
 \tc110
 \tc000
 \tc150
 \tc040
 \tc060
 \tc601
 \tc511
 \tc401
 \tc431
 \tc321
 \tc211
 \tc101
 \tc341
 \tc231
 \tc121
 \tc011
 \tc422
 \tc312
 \tc202
 \tc051
 \tc240
 \tc130
 \tc020
 \tc141
 \tc031
}
\end{scope}
\end{tikzpicture}
}

\newcommand{\picturefamilies}{
    \begin{tikzpicture}[]
 \begin{scope}[shift={(4.5,0)},scale=0.5]
\def\x{6,6,6,6}
\foreach \i [count=\xi] in \x{
    \draw [step=1.0,blue, very thick] (\xi-1,0) grid (\xi,-\i);
}
 \end{scope}
 
  \begin{scope}[shift={(8.5,0)},scale=0.5]
\def\x{6,1,1,1}
\foreach \i [count=\xi] in \x{
    \draw [step=1.0,blue, very thick] (\xi-1,0) grid (\xi,-\i);
}
 \end{scope}
 
  \begin{scope}[shift={(0,0)},scale=0.5]
\def\x{6,5,4,3,2,1}
\foreach \i [count=\xi] in \x{
    \draw [step=1.0,blue, very thick] (\xi-1,0) grid (\xi,-\i);
}
 \end{scope}

\end{tikzpicture}
}

\newcommand{\picturecolumns}{
\begin{tikzpicture}[scale=0.5]
  \begin{scope}[shift={(0,0)}]
    \node[anchor=east] at (-1,-0.5) {\begin{tabular}{c}
         \cref{lem:1rowgame}
    \end{tabular}};
    \node at (5.5,-0.5) {1};
    \node at (4.5,-0.5) {2};
    \node at (3.5,-0.5) {1};
    \node at (2.5,-0.5) {2};
    \node at (1.5,-0.5) {1};
    \node at (0.5,-0.5) {$\cdots$};
\def\x{1,1,1,1,1,1}
\foreach \i [count=\xi] in \x{
    \draw [step=1.0,blue, very thick] (\xi-1,0) grid (\xi,-\i);
}
 \end{scope}
  \begin{scope}[shift={(16,0)}]
    \node[anchor=east] at (-1,-1) {\begin{tabular}{c}
    \cref{lem:2rowgame} \\ (case 1)
    \end{tabular}};
    \node at (5.5,-0.5) {0};
    \node at (4.5,-1.5) {1};
    \node at (3.5,-1.5) {2};
    \node at (2.5,-1.5) {1};
    \node at (1.5,-1.5) {2};
    \node at (0.5,-1.5) {$\cdots$};    
\def\x{2,2,2,2,2}
\foreach \i [count=\xi] in \x{
    \draw [step=1.0,blue, very thick] (\xi-1,0) grid (\xi,-\i);
}
 \end{scope}
  \begin{scope}[shift={(16,-3)}]
    \node[anchor=east] at (-1,-1) {\begin{tabular}{c}
    \cref{lem:2rowgame} \\ (case 2)
    \end{tabular}};
    \node at (4.5,-0.5) {$\bcancel{0}$};
    \node at (3.5,-1.5) {1};
    \node at (2.5,-1.5) {2};
    \node at (1.5,-1.5) {1};
    \node at (0.5,-1.5) {$\cdots$};
\def\x{2,2,2,2,1,1}
\foreach \i [count=\xi] in \x{
    \draw [step=1.0,blue, very thick] (\xi-1,0) grid (\xi,-\i);
}
 \end{scope}
   \begin{scope}[shift={(0,-2)}]
    \node[anchor=east] at (-1,-1.5) {\begin{tabular}{c}
    \cref{lem:3rowgame} \\ (case 1)
    \end{tabular}};
    \node at (5.5,-0.5) {0};
    \node at (4.5,-1.5) {2};
    \node at (3.5,-1.5) {0};
    \node at (2.5,-1.5) {2};
    \node at (1.5,-1.5) {0};
    \node at (0.5,-1.5) {$\cdots$};
\def\x{3,3,3,3,3}
\foreach \i [count=\xi] in \x{
    \draw [step=1.0,blue, very thick] (\xi-1,0) grid (\xi,-\i);
}
 \end{scope}
    \begin{scope}[shift={(0,-6)}]
    \node[anchor=east] at (-1,-1.5) {\begin{tabular}{c}
    \cref{lem:3rowgame} \\ (case 2)
    \end{tabular}};
    \node at (4.5,-0.5) {$\bcancel{0}$};
    \node at (3.5,-1.5) {2};
    \node at (2.5,-1.5) {0};
    \node at (1.5,-1.5) {2};
    \node at (0.5,-1.5) {$\cdots$};
\def\x{3,3,3,3,1,1}
\foreach \i [count=\xi] in \x{
    \draw [step=1.0,blue, very thick] (\xi-1,0) grid (\xi,-\i);
}
 \end{scope}
\begin{scope}[shift={(0,-10)}]
    \node[anchor=east] at (-1,-1.5) {\begin{tabular}{c}
    \cref{lem:3rowgame} \\ (case 3)
    \end{tabular}};
    \node at (4.5,-0.5) {$\bcancel{1}$};
    \node at (3.5,-1.5) {0};
    \node at (2.5,-1.5) {1};
    \node at (1.5,-1.5) {0};
    \node at (0.5,-1.5) {$\cdots$};
\def\x{3,3,3,3,2,2}
\foreach \i [count=\xi] in \x{
    \draw [step=1.0,blue, very thick] (\xi-1,0) grid (\xi,-\i);
}
 \end{scope}
 \begin{scope}[shift={(16,-6)}]
    \node[anchor=east] at (-1,-1.5) {\begin{tabular}{c}
    \cref{lem:3rowgame} \\ (case 4)
    \end{tabular}};
    \node at (4.5,-0.5) {0};
     \node at (3.5,-0.5) {1};
      \node at (2.5,-0.5) {0};
    \node at (1.5,-0.5) {$\cdots$};
    \node at (0.5,-1.5) {0};

\def\x{3,2,2,2,2,1}
\foreach \i [count=\xi] in \x{
    \draw [step=1.0,blue, very thick] (\xi-1,0) grid (\xi,-\i);
}
 \end{scope}
 \begin{scope}[shift={(16,-10)}]
    \node[anchor=east] at (-1,-1.5) {\begin{tabular}{c}
    \cref{lem:3rowgame} \\ (case 5)
    \end{tabular}};
    \node at (2.5,-0.5) {0};
    \node at (3.5,-0.5) {$\bcancel{0}$};
    \node at (3.5,-1.5) {0};
    \node at (2.5,-1.5) {1};
    \node at (1.5,-1.5) {0};
    \node at (0.5,-1.5) {$\cdots$};
\def\x{3,3,3,3,2,1}
\foreach \i [count=\xi] in \x{
    \draw [step=1.0,blue, very thick] (\xi-1,0) grid (\xi,-\i);
}
 \end{scope}
\end{tikzpicture}
}

\newcommand{\ericferrer}{
\begin{tikzpicture}[scale=0.5, every node/.style={scale=1}]
 \begin{scope}[shift={(0,0)},scale=1]
\def\x{6,4,3,3,1,1}
\foreach \i [count=\xi] in \x{
    \draw [step=1.0,blue, very thick] (\xi-1,0) grid (\xi,-\i);
}
\end{scope}
\end{tikzpicture}
}

\newcommand{\picturecut}{
\begin{tikzpicture}[scale=0.6, every node/.style={scale=1}]

\begin{scope}[shift={(18,0)},scale=1]
\begin{scope}[shift={(0,0)},scale=1]
\def\x{3,3,3}
\foreach \i [count=\xi] in \x{
    \draw [step=1.0,blue, very thick] (\xi-1,0) grid (\xi,-\i);
}
\end{scope}
\begin{scope}[shift={(3.5,0)},scale=1]
\node at (0.5,-0.5) {$\scriptstyle\alpha_{03}$};
\node at (0.5,-1.5) {$\scriptstyle\alpha_{13}$};
\node at (0.5,-2.5) {$\scriptstyle\alpha_{23}$};
\def\x{2,2,1}
\foreach \i [count=\xi] in \x{
    \draw [step=1.0,blue, very thick] (\xi-1,0) grid (\xi,-\i);
}
\end{scope}
\begin{scope}[shift={(0,-3.5)},scale=1]
\node at (0.5,-0.5)
{$\scriptstyle\alpha_{30}$};
\node at (1.5,-0.5) {$\scriptstyle\alpha_{31}$};
\node at (2.5,-0.5) {$\scriptstyle\alpha_{32}$};
\def\x{2,2,1}
\foreach \i [count=\xi] in \x{
    \draw [step=1.0,blue, very thick] (\xi-1,0) grid (\xi,-\i);
}
\end{scope}
\end{scope}
\begin{scope}[shift={(8,0)},scale=1]
\begin{scope}[shift={(0,0)},scale=1]
\def\x{2,2}
\foreach \i [count=\xi] in \x{
    \draw [step=1.0,blue, very thick] (\xi-1,0) grid (\xi,-\i);
}
\end{scope}
\begin{scope}[shift={(2.5,0)},scale=1]
\def\x{2,1,1}
\foreach \i [count=\xi] in \x{
    \draw [step=1.0,blue, very thick] (\xi-1,0) grid (\xi,-\i);
}
\end{scope}
\begin{scope}[shift={(0,-2.5)},scale=1]
\def\x{3,2}
\foreach \i [count=\xi] in \x{
    \draw [step=1.0,blue, very thick] (\xi-1,0) grid (\xi,-\i);
}
\end{scope}
\end{scope}
\begin{scope}[shift={(0,0)},scale=1]
\begin{scope}[shift={(0,0)},scale=1]
\def\x{1}
\foreach \i [count=\xi] in \x{
    \draw [step=1.0,blue, very thick] (\xi-1,0) grid (\xi,-\i);
}
\end{scope}
\begin{scope}[shift={(1,0)},scale=1]
\def\x{1,1}
\foreach \i [count=\xi] in \x{
    \draw [step=1.0,blue, very thick] (\xi-1,0) grid (\xi,-\i);
}
\end{scope}
\begin{scope}[shift={(0,-1)},scale=1]
\def\x{4}
\foreach \i [count=\xi] in \x{
    \draw [step=1.0,blue, very thick] (\xi-1,0) grid (\xi,-\i);
}
\end{scope}
\end{scope}
\end{tikzpicture}
}
\newcommand{\picturebigcut}{
\begin{tikzpicture}[scale=0.6, every node/.style={scale=1}]
\begin{scope}[shift={(0,0)},scale=1]
\begin{scope}[shift={(-2.5,2.5)},scale=1]
\gb00
\gb11
\def\x{8,7,2,2,2,2,2,2,2,2,2,1}
\foreach \i [count=\xi] in \x{
    \draw [step=1.0,blue, very thick] (\xi-1,0) grid (\xi,-\i);
}
\end{scope}

\begin{scope}[shift={(0,0)},scale=1]
\def\x{3,3,3}
\gb00
\foreach \i [count=\xi] in \x{
    \draw [step=1.0,blue, very thick] (\xi-1,0) grid (\xi,-\i);
}
\end{scope}
\begin{scope}[shift={(3.5,0)},scale=1]
\node at (0.5,-0.5) {$\scriptstyle\alpha_{03}$};
\node at (0.5,-1.5) {$\scriptstyle\alpha_{13}$};
\node at (0.5,-2.5) {$\scriptstyle\alpha_{23}$};
\def\x{3,2,1,1}
\foreach \i [count=\xi] in \x{
    \draw [step=1.0,blue, very thick] (\xi-1,0) grid (\xi,-\i);
}
\end{scope}
\begin{scope}[shift={(0,-3.5)},scale=1]
\node at (0.5,-0.5)
{$\scriptstyle\alpha_{30}$};
\node at (1.5,-0.5) {$\scriptstyle\alpha_{31}$};
\node at (2.5,-0.5) {$\scriptstyle\alpha_{32}$};
\def\x{2,1,1}
\foreach \i [count=\xi] in \x{
    \draw [step=1.0,blue, very thick] (\xi-1,0) grid (\xi,-\i);
}
\end{scope}
\end{scope}
\begin{scope}[shift={(14.5,0)},scale=1]
\begin{scope}[shift={(-2.5,2.5)},scale=1]
\gb00
\gb11
\gb22
\gb33
\def\x{8,7,7,6,4,4,4,3,3,2,1}
\foreach \i [count=\xi] in \x{
    \draw [step=1.0,blue, very thick] (\xi-1,0) grid (\xi,-\i);
}
\end{scope}

\begin{scope}[shift={(2,-2)},scale=1]
\def\x{2,1}
\gb00
\foreach \i [count=\xi] in \x{
    \draw [step=1.0,blue, very thick] (\xi-1,0) grid (\xi,-\i);
}

\end{scope}
\end{scope}
\end{tikzpicture}
}
\journal{\ifthenelse{\arx=1}{arXiv}{Discrete Applied Mathematics}}

\newcommand{\picsgandmiseredr}{
\tikzset{every node/.style={scale=0.7}}
\begin{tikzpicture}[scale=0.7]
\begin{scope}[shift={(6,-1)},scale=1]
\def\x{5,4,2,1,1}
\foreach \i [count=\xi] in \x{
    \draw [step=1.0,blue, very thick] (\xi-1,0) grid (\xi,-\i);
}
\node at (1.5,-0.5){(0,0)};
\node at (2.5,-0.5){(2,2)};
\node at (3.5,-0.5){(1,0)};
\node at (4.5,-0.5){(0,1)};
\node at (1.5,-1.5){(2,2)};
\node at (2.5,-1.5){(0,1)};
\node at (1.5,-2.5){(1,0)};
\node at (1.5,-3.5){(0,1)};
\node at (0.5,-4.5){(0,1)};
\node at (0.5,-3.5){(1,0)};
\node at (0.5,-2.5){(0,1)};
\node at (0.5,-1.5){(1,0)};
\node at (0.5,-0.5){(2,1)};
\end{scope}
\end{tikzpicture}
}

\newcommand{\picsgandmiserelctr}{
\tikzset{every node/.style={scale=0.7}}
\begin{tikzpicture}[scale=0.7 ]

\begin{scope}[shift={(0,-1)},scale=1]
\def\x{4,3,3,3,3}
\foreach \i [count=\xi] in \x{
    \draw [step=1.0,blue, very thick] (\xi-1,0) grid (\xi,-\i);
}
\node at (0.5,-4.5){(0,1)};
\node at (5.5,-1.5){(0,1)};
\node at (5.5,-2.5){(0,1)};
\node at (5.5,-0.5){(0,1)};
\node at (1.5,-3.5){(0,1)};
\node at (2.5,-3.5){(0,1)};
\node at (3.5,-3.5){(0,1)};
\node at (4.5,-3.5){(0,1)};
\node at (4.5,-2.5){(1,0)};
\node at (3.5,-2.5){(2,2)};
\node at (4.5,-1.5){(2,2)};
\node at (3.5,-1.5){(0,0)};
\node at (2.5,-2.5){(1,0)};
\node at (2.5,-1.5){(2,1)};
\node at (4.5,-0.5){(1,0)};
\node at (3.5,-0.5){(2,1)};
\node at (2.5,-0.5){(0,0)};
\node at (1.5,-0.5){(1,1)};
\node at (1.5,-1.5){(0,0)};
\node at (1.5,-2.5){(2,2)};
\node at (0.5,-3.5){(1,0)};
\node at (0.5,-2.5){(0,1)};
\node at (0.5,-1.5){(1,2)};
\node at (0.5,-0.5){(0,0)};
\end{scope}
\end{tikzpicture}
}

\newcommand{\picturegurvich}{
\setlength{\unitlength}{1cm}
\begin{picture}(8, 5)(-3,.6)
\put(.5,2.8){\color{red}\oval(8,  1)}
\put(.5,2.6){\color{green}\oval(9,  1.7)}
\put(.5,2.4){\color{blue}\oval(10,  2.4)}
\put(-0.8,2.6){\color{red}{ pet games}}
\put(-1,  1.9){\color{green}{tame games}}
\put(-1.5,  1.4){\color{blue}{ domestic games}}

\put(.5,3.8){\color{gray}\oval(7, 1.5)}
\put(-1,4){\color{gray}{forced games}}

\put(.5,3.8){\color{cyan}\oval(8.5, 3.1)}
\put(-1.3,4.8){\color{cyan}{returnable games}}

\end{picture}
}
\begin{document}

\begin{frontmatter}

\title{\ourtitle}
%BIUS

\author[inst1]{Eric Gottlieb }
\ead{gottlieb@rhodes.edu}
\affiliation[inst1]{organization={Rhodes College},%Department and Organization
            %addressline={Address One}, 
            %city={City One},
            %postcode={00000}, 
            state={Tennessee},
            country={ U.S.A.}}

\author[inst2]{Matjaž Krnc}
\ead{matjaz.krnc@upr.si}
\author[inst2]{Peter Muršič}
\ead{peter.mursic@famnit.upr.si}

\affiliation[inst2]{organization={Faculty of Mathematics, Natural Sciences and Information Technologies},%Department and Organization
            %addressline={Address Two}, 
            city={Koper},
            %postcode={22222}, 
            %state={},
            country={Slovenia}}

% \affiliation[inst3]{organization={Department Two},%Department and Organization
%             %addressline={Address Two}, 
%             city={Koper},
%             %postcode={22222}, 
%             %state={State Two},
%             country={Slovenia}}

\begin{abstract}
% Given an integer partition  of $n$, we consider the impartial combinatorial game LCTR in which moves consist of removing either the left column or top row of its \Young. 
% We show that for both normal and misère play, the optimal strategy can consist mostly of mirroring the opponent's moves. 
% This allows for computing the Sprague-Grundy value of the given game in $O(\log(n))$ time, where time unit allows for reading an integer, or performing a basic arithmetic operation.
% We improve the previously known bound due to Ilić (2019) of $O(n)$ time by using a structural approach in place of Sprague-Grundy recursion. 
% We also establish that both LCTR and \King{} are domestic  as well as returnable, and on the other hand neither tame nor forced.

Given an integer partition  of $n$, we consider the impartial combinatorial game LCTR in which moves consist of removing either the left column or top row of its \Young. 
We show that for both normal and misère play, the optimal strategy can consist mostly of mirroring the opponent's moves. 
We also establish that both LCTR and \King{} are domestic as well as returnable, and on the other hand neither tame nor forced.

For both games, those structural observations allow for computing the Sprague-Grundy value any position in $O(\log(n))$ time, assuming that the time unit allows for reading an integer, or performing a basic arithmetic operation.
This improves on the previously known bound of $O(n)$ due to Ilić (2019). 
We also cover some other complexity measures of both games, such as state-space complexity, and number of leaves and nodes in the corresponding game tree.
\end{abstract}

%%Graphical abstract
% \begin{graphicalabstract}
% \includegraphics{grabs}
% \end{graphicalabstract}

%%Research highlights
\ifthenelse{\arx=1}{}{
\begin{highlights}
\item We study the impartial combinatorial game LCTR in which moves consist of removing either the left column or top row of a \Young{}.

\item We show that for both normal and misère play, the optimal strategy can consist mostly of mirroring the opponent's moves. 

\item We establish that both LCTR and \King{} are domestic  as well as returnable, and on the other hand neither tame nor forced.

\item We describe an algorithm 
to compute the Sprague-Grundy value 
%of the given game 
in $O(\log(n))$ time, and also cover some other complexity measures of both games, such as state-space complexity, and number of leaves and nodes in the corresponding game tree.

\end{highlights}
}


\end{frontmatter}

\section{Introduction and background}
Combinatorial game theory is a large and growing field that includes in its scope a wide range of game types, generally focusing on two-player games in which both players have perfect information and there are no moves of chance.
We restrict ourselves to two-player combinatorial games that end in finitely many moves.
Every game ends in one of the two possible outcomes: one of the players wins, while the other loses. 
Under \emph{normal play}, the winner is the player who makes the last move, leaving their opponent with no move. Under \emph{misère} rules, the win condition is reversed.

The main question in combinatorial game theory is, given some position and optimal play, which player has a winning strategy? An $\N$-position (winning position) is one in which the player whose turn it is can guarantee a win. 
The complementary positions are $\P$-positions (losing positions) in which the player that goes second can guarantee a win.
The game is effectively solved if we can classify all of its positions as either $\N$ or $\P$. 
Sprague \cite{Spr35,Spr37} and  Grundy \cite{Gru39} introduced a method of quantifying game positions for normal play \emph{impartial} games, i.e., those in which both players
have the same possible moves in each position. These \emph{Sprague-Grundy values} are a generalization of winning and losing positions. 
Furthermore, these values are instrumental in the theory of \emph{disjunctive sums of impartial games}.
This is why recent research on impartial games is predominantly focused on the Sprague-Grundy values rather than on $\P/\N$-positions.
For more information on combinatorial game theory, including impartial games and disjunctive sums, see \cite{berlekamp2017winning1,berlekamp2017winning2,berlekamp2017winning3,berlekamp2017winning4, Con76,Sie13}.

Over the last half century, a number of researchers including Dailly et al~\cite{DAILLY2020509}, Abuku and Tada~\cite{MR4556411}, Irie~\cite{irie2018p}, and Sato~\cite{sato1970maya} have investigated games on (integer) partitions. In this paper we study the impartial finite game introduced in \cite{Ilic2019}, where two players take turns removing from a given \Young{} $\lambda$ either leftmost column or topmost row, hence the name LCTR.
We show that for both normal and misère play of LCTR, the optimal strategy can consist mostly of mirroring the opponent's moves. 
In terms of structural classification of normal-misère game pairs as introduced in Gurvich et al. \cite{GURVICH201854} we
establish that both games are domestic as well as returnable, and on the other hand neither tame nor forced.

In \cite{Ilic2019}, the authors studied the behaviour of some specifically structured sub-families of LCTR games (see \cref{lem:ilicgama}) and devised a procedure deciding which player has the winning strategy on a given \Young{} and determining the Sprague-Grundy value associated to the \Young{} in $O(n)$ time units \cite{jelena_code}\footnote{Here time unit allows performing a basic arithmetic operation (see \cref{sec:timecomp}), while $n$ is defined  so that $\lambda$ is a partition of $n$.}.
For both games, our structural insights allow for computing the Sprague-Grundy value any position in $O(\log(n))$ time  units.
We also cover some other complexity measures of both games, such as state-space complexity, and number of leaves and nodes in the corresponding game tree.

In \cref{sec:prelims} we cover the basic definitions needed for the study of LCTR and related games. 
In \cref{sec:LCTRKING} we establish the properties that are common to both normal and misère play of LCTR.
In \cref{sec:LCTR misere,sec:LCTR} 
we take the study of the normal form of LCTR, as well as the normal form of a game called \emph{\King{}} that is closely related to the misère form of LCTR, and show that for both games, the optimal strategy can consist mostly of mirroring the opponent's moves.
In \cref{sec:hierarchy} we establish a structural classification of both LCTR and \King{} within the  framework of  Gurvich et al. 
\Cref{sec:computing} focuses on computing the exact Sprague-Grundy values for both games in $O(\log n)$ time units.
In the same section we also cover some other complexity measures of both games, such as state-space complexity, and number of leaves and nodes in the corresponding game tree.
Finally, we discuss future research directions as well as some related open problems in \cref{sec:open}.

\section{Preliminaries}\label{sec:prelims}

In this section, we provide background information on partitions and \Youngs{}, impartial games, Sprague-Grundy theory, and complexity theory.

\subsection{\Youngs{} and Partitions}\label{sec:partitions}

A \emph{partition} of a nonnegative integer $n$ is a sequence $\lambda = \br{\lambda_1, \dots, \lambda_r}$ of integers, where $\lambda_1 \geq \cdots \geq \lambda_r > 0$ and $\lambda_1 + \cdots + \lambda_r = n$. 
We also write $\lambda \vdash n$ and we say that $\lambda$ is a partition of $n$. 
The $\lambda_j$'s are called the \emph{parts} of $\lambda$. When some of the parts of a partition are equal, we may abbreviate $\br{\lambda_1^{m_1}, \dots, \lambda_k^{m_k}}$, where $\lambda_j$ appears $m_j\ge0$ times. 
The \emph{\Young{}} of $\lambda$ is a left-justified array of boxes in which the $j$th row from the top contains $\lambda_j$ boxes. 
Partitions are drawn as \Youngs{}%
, and no distinction is made between a partition and its \Young{}. 
The \Young{} of $\br{6, 4, 3, 3, 1, 1} = \br{6, 4, 3^2, 1^2}$ is shown in \cref{fig:ferrer}. 
The only partition of 0 is the empty sum, which we write as $\br{}$. %
 We define the conjugate partition of $\lambda$ as $\lambda'=\br{\lambda'_1,\ldots,\lambda'_{\lambda_1}}$ 
where $\lambda'_i$ is the largest integer such that $\lambda_{\lambda'_i}\geq i$.

\begin{figure}[h]
\centering
\ericferrer
\caption{The \Young{} of $\br{6, 4^2, 2, 1^2}$ }\label{fig:ferrer}
\end{figure}

If $i$ and $j$ are such that $\lambda_i > 0$ and $0<j \leq \lambda_i$ then the box in the $i$th row and $j$th column of the \Young{} is called the $(i, j)$th box of the \Young{}.
The \emph{subpartition} of $\lambda$ corresponding to the \Young{} rooted at the $(i, j)$th box of $\lambda$ is $\lambda[i-1, j-1] = \br{\lambda_i - j+1, \lambda_{i+1} - j+1, \dots}$, where trailing nonpositive entries are removed.
We define a \emph{corner} of a partition $\lambda$ to be a pair $(i, j)$ such that $\lambda[i,j] = \br{1}$. If $(i, j)$ is a corner of $\lambda$ then we define $\lambda[i, j]$ to be a \emph{corner subpartition} of $\lambda$. %
If $i$ and $j$ are positive integers not satisfying the above conditions, then we define  $\lambda[i-1,j-1]$ to be the empty partition $\br{}$.
Thus $\lambda[0, 0] = \lambda$ for any partition.
A key property of the subpartitions is that for all non-negative integers $i_1,i_2,j_1,j_2$ we have the identity  
\begin{align}\label{eq:chainsubpartition}
(\lambda[i_1,j_1])[i_2,j_2]=\lambda[i_1+i_2,j_1+j_2].
\end{align}
The \emph{Durfee length} of partition $\lambda=\br{\lambda_1,\lambda_2,\dots}$ is defined as
$
\D(\lambda)=\max \{\ell \mid \lambda_\ell\ge \ell \}
$.
Let $r$ and $c$ be positive integers. 
The \emph{staircase}, \emph{rectangle}, and \emph{gamma} families of partitions are defined as 
\begin{align*}
    \staircase_r&=\br{r,r-1,r-2,\ldots, 1}&
    \rectangle_{r,c}&= \br{c^r}&
    \gama_{r,c}&= \br{c, 1^{r-1}},
\end{align*}
respectively.

\begin{figure}
    \centering
 \picturefamilies
    \caption{%
    Partitions corresponding to $\protect\staircase_{6}$ and  $\protect\rectangle_{6,4}$ and $\protect\gama_{6,4}$
    }
    \label{fig:symgames}
\end{figure}

\subsection{Impartial games and Sprague-Grundy values}
\label{sec:imgames}\label{def:PN-recursion}

In this paper our focus will be exclusively on impartial finite combinatorial games, which we will refer to simply as \emph{games}.
The positions in the games we study are partitions or \Youngs{}, which we describe in the following subsection. 
A \emph{subposition} of a game is a position that can be reached after a nonnegative number of %
moves. Every game begins in some position, so we sometimes use the terms interchangeably. 

A \emph{terminal position} is a position from which no further moves are possible. 
Under normal play of a combinatorial game, a player who reaches a terminal position is the winner. 
Under misère play, a player who reaches a terminal position is the loser. 
Throughout this paper, we assume normal play unless otherwise indicated.

Let $G$ be a game under normal (misère) play. 
We define the set of its winning positions $\N_G$, 
and the set of its losing positions $\P_G$, as follows:
\begin{enumerate}
    \item Terminal positions are in $\P_G$ ($\N_G$). \label{def:PN-recursion1}
    \item A position is in $\P_G$, if none of its moves leads to a position in $\P_G$.
    \label{def:PN-recursion2}
    \item A position is in $\N_G$, if at least one of its moves leads to a position in $\P_G$.
    \label{def:PN-recursion3}
\end{enumerate}
We denote the above sets simply as $\P$ and $\N$, whenever the game in question is clear from the context.

A two player combinatorial game is \emph{impartial} if both players have the same possible moves in each position.
Let $A$ be a position of such a game $G$. 
If there is a move from $A$ to position $B$, we write $A \to B$. 
For a finite set $S$ of nonnegative integers the {\it minimum excluded value} of $S$
is defined as the smallest nonnegative integer that is not in $S$ and
is denoted by  $\Mex (S)$. 
The \emph{Sprague-Grundy} value ($\G$ value) of $A$ is defined recursively by  
\begin{align}
\label{eq:main_SG}
\G(A) = \Mex (\{\G(B) \mid A \to B\}).
\end{align}

A position of $\G$ value $t$ is called a {\it $t$-position}. 
For the case of terminal positions we have $\Mex(\emptyset)=0$, hence such positions have $\G$ value $0$.

If $G_1$ and $G_2$ are games, their \emph{disjunctive sum} $G_1 + G_2$ is defined to be the game in which players in turn choose a summand ($G_1$ or $G_2$) and make a move in that game. Play continues until terminal positions are reached in both summands. It is well known that $\G(G_1) = \G(G_2)$ if and only if $G_1+G_2 \in \P$ (see \cite{Con76}). 

$\G$ values provide a useful mechanism for determining winning and losing positions for impartial games under normal play.
In particular, the $\P$-positions are precisely those that have $\G$ value $0$ while the $\N$-positions are those with nonzero $\G$ values.
The Sprague-Grundy  theorem \cite{Sie13} shows that given the sum of impartial games under normal play, we can replace every game with a pile of Nim, with the size corresponding to the $\G$ value of the game.
Doing so will not change the $\P/\N$-positions of the sum.

For normal play, one can partition the positions to equivalence classes with respect to their $\G$ value. 
For misère play such equivalence classes, for which we can replace one game with another within the sum, are different and far more difficult to characterize in general.
The equivalence classes for misère are called misère quotients and form a commutative monoid \cite{PLAMBECK2008593}. Hence while one can study misère quotients, one can also focus on just $\P/\N$-positions.

We make use of the following corollary, which follows from the definition of $\G$ value. For detailed coverage of combinatorial game theory we refer the reader to \cite{Sie13}.
\begin{corollary}
\label{cor:adjacentMex}
Let $A,B$ be positions such that $A\to B$.
Then $\G(A)\neq \G(B)$.
\end{corollary}

\subsection{Tame, pet, domestic, and miserable impartial games}

As early as in 1956 Grundy and Smith  \cite{grundy_smith_1956}  noticed that
playing a game under the mis\`{e}re convention may be difficult in general.
In \cite{GURVICH201854} Gurvich and Ho focus on the exceptions, that is, on the games
for which functions  $\G$ and $\G^-$ are closely related.

The mis\`{e}re $\SG$ value $\G^-(x)$  of a position $x$ in a game  $G$
is defined by the same recursion (\ref{eq:main_SG}), but the initialization is different:
$\G^-(x) = 1$  (rather than $\G^-(x) = 0$) for all terminal positions.

A position $x$  will be called an $(i,j)$-position if  $\G(x) = i$  and  $\G^-(x) = j$. 
We say that $(\G(x),\G^-(x))$ is a  Sprague-Grundy pair of $x$.
By definition, every terminal position is a $(0,1)$-position.

\begin{definition} \label{D.DTP}
An impartial game will be called
\begin{enumerate} [(i)] 
\item {\em domestic} if it has neither $(0,k)$-positions nor $(k,0)$-positions with $k \geq 2$;
\item {\em tame} if it has only $(0,1)$-positions, $(1,0)$-positions, and $(k,k)$-positions with $k \geq 0$;
\item {\em pet} if it has only $(0,1)$-positions, $(1,0)$-positions, and $(k,k)$-positions with $k \geq 2$.
\item {\em forced} if each move from a $(0,1)$-position results in a $(1,0)$-position and vice versa;
\item {\em returnable} if the following, weaker, implications hold:
let $x$ be a $(0,1)$-position (resp.,~a $(1,0)$-position) movable to a non-terminal position $y$, then
$y$ is movable to a $(0,1)$-position (resp.,~to a $(1,0)$-position).
\end{enumerate}
\end{definition}

Tame games were introduced by Conway~\cite[Chapter 12]{Con76}
(see page 178), pet games were introduced recently by Gurvich et al.$\!$ in the preprints \cite{Gur11, Gur120},
while domestic games are introduced by Gurvich et al.$\!$~\cite{GURVICH201854}.
According to the above definitions, domestic, tame, pet, returnable and forced games form nested classes:
a pet game is tame, a tame game is domestic and a forced game is returnable; see \cref{fig:picturegurvich}.

\begin{figure}
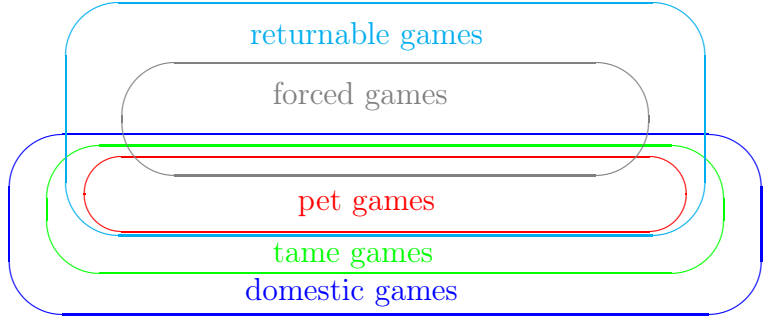

    \centering
    \picturegurvich
    \caption{Hierarchy of impartial games as shown in \cite{GURVICH201854}}
\label{fig:picturegurvich}
\end{figure}

\subsection{Complexity theory}

We provide some technical definitions regarding complexity in the context of impartial games that will be used later. %
Let $G$ be an impartial game and let $p$ be a fixed starting position, which we sometimes refer to as $G(p)$. 
The corresponding \emph{game tree} $T_p$ is a rooted tree obtained by the following procedure.
\begin{itemize}
    \item The original position $p$ is the root of $T_p$. 
    \item For each subposition $\tilde p$ 
    with  $p\to \tilde p$,
    there is a directed edge from $p$ to the root of $T_{\tilde p}$. 
\end{itemize} 
Note that multiple copies of the same position may occur in the game tree. 
These are treated as distinct nodes. 
We distinguish several parameters related to $G$ and $p$:
\begin{enumerate}
    \item The number of leaves of the game tree $T_p$ is denoted by $\leaf(T_p)$, while the number of nodes of the game tree $T_p$ is denoted by $\nodes(T_p)$. 
    We extend those notations by defining $\leaf(G(p))=\leaf(T_p)$, and similarly $\nodes(G(p))=\nodes(T_p)$.

    \item The \emph{height} of $T_p$ is the length of the longest path from $p$.
    Provided that $T_p$ has finite height, the height of any proper subtree of $T_p$ is at least one less than the height of $T_p$.

    \item 
    We say that subpositions $\tilde p$ and $\bar p$ of $p$ are \emph{distinct} if $T_{\tilde p}$ is not isomorphic to $T_{\bar p}$.
    The \emph{state-space complexity} of a game is the number of pairwise distinct game positions reachable from the initial position $p$ of the game.
\end{enumerate}

A game $G$ is said to be \emph{finite} if $T_p$ has finite height for any position $p$ of $G$.
The \emph{computational complexity} of a game is the time required to decide whether a given initial position is winning.

\section{LCTR and \King{}}\label{sec:LCTRKING}

We now describe the core objects of this paper, the games LCTR and \King{}. After providing definitions in \cref{sec:lctr,sec:kingintro}, we mention several important properties of both games in \cref{sec:obsLCTRking}.

\subsection{LCTR}\label{sec:lctr}
A position in the \emph{leftmost column---topmost row}, or LCTR game (see \cite{Ilic2019}), is denoted by $\L(\lambda)$, where $\lambda$ is a partition, i.e., a non-increasing list of $r$ positive integers $\lambda=\br{\lambda_1, \lambda_2, \ldots, \lambda_r}$ with $\lambda_1\geq \lambda_2 \geq \ldots \geq \lambda_r$. 

The \emph{ terminal position} occurs when $r = 0$, which we denote by $\L(\br{})$.
Let $A\neq \L(\br{})$ be a non-terminal position.
There are two possible moves a player can make, leading to a new position $\tilde A$, namely:
\begin{itemize}
    \item $\tilde A=\L(\lambda[1,0])$, or
    \item $\tilde A=\L(\lambda[0,1])$.
\end{itemize}

Given two non-negative integers 
$i < r$ and $j\le \lambda_{i+1}$, or $(i,j)=(r,0)$,
we denote the subposition reached from $A$ by $i$ top row moves and $j$ left column moves as $A[i,j]$, formally 
$A[i,j]=\L(\lambda[i,j]).$

\subsection{\King{}}\label{sec:kingintro}
We now introduce a new game named \emph{\King{}}, %
 which is simpler than LCTR and is a good introduction to the study of LCTR.
Normal play in \King{} is equivalent to misère play in LCTR in a way that is made more precise in \cref{sec:LCTR misere}. 
In \King{}, a ``restricted'' rook chess piece begins in the top-left box of a \Young{}. 
The rook can move a single adjacent box down or right. 
Under normal play, then, the player who gets the rook to a corner, where no moves remain, is the winner. 
There are similarities with Impartial chess played with a King \cite{impchess,impchess1}, but in \King{} we cannot move diagonally and the board is a (not necessarily rectangular) partition.

Let us define the game of \King{} $\K(\lambda)$, for a given partition $\lambda=\br{\lambda_1, \lambda_2, \ldots, \lambda_r}$. 
There are at most two moves a player can make. 
In particular, the move resulting in the subposition $\K(\lambda[1,0])$ is allowed whenever $r>1$, the move resulting in subposition $\K(\lambda[0,1])$ is allowed whenever $\lambda_1>1$.
Given two non-negative integers 
$i<r$ and $j<\lambda_{i+1}$,
we denote the subposition reached from $\B$ by the rook moving $j$ boxes to the right and $i$ boxes down as $\B[i,j]$, formally 
 $\B[i,j]=\K(\lambda[i,j])$. 

\begin{figure}[h]
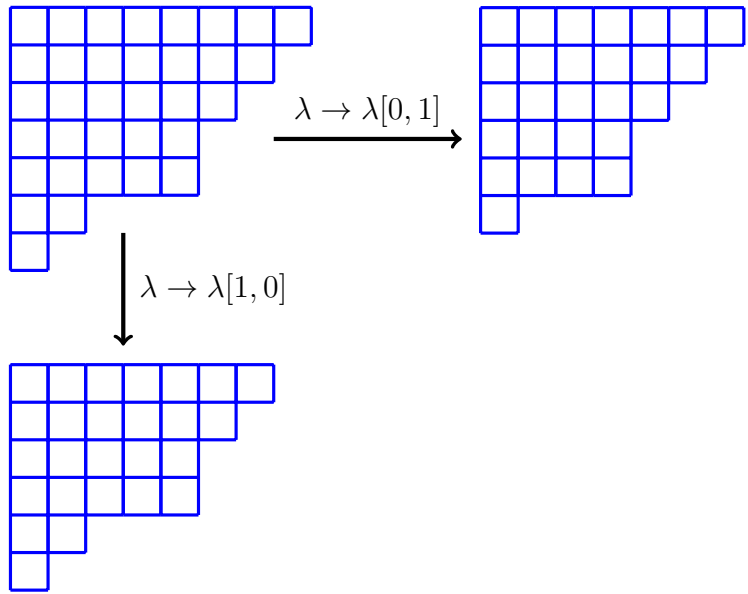

\begin{center}
\picturemex
\caption{The partition $\lambda=\br{8,7,6,5^2,2,1}$ along with two subpartitions $\lambda[1,0]=\br{7,6,5^2,2,1}$ and  $\lambda[0,1]=\br{7, 6, 5, 4^2, 1}$ that correspond to moves in LCTR or \King{}.\label{fig:MexRecursion}
}
\end{center}
\end{figure}

\subsection{Observations regarding LCTR and \King{}}\label{sec:obsLCTRking}

Let $\lambda$ be a partition, let $i,j$ be nonnegative integers such that $\lambda[i,j]$ is defined and non-empty.
Within the game tree of either $\L(\lambda)$ or $\K(\lambda)$, a position corresponding to $\lambda[i,j]$ occurs several times%
\footnote{There are exactly $\binom{i+j}{j}$ many ways to make $i$ down moves and $j$ right moves to get to $\lambda[i,j]$.
However, this partition may be equal to $\lambda[i', j']$ with $i \neq i'$ or $j\neq j'$, so the best we can say is that the partition $\lambda[i,j]$ occurs at least $\binom{i+j}{j}$ times.  
} but its $\G$ value need only be computed once. 
Furthermore, in either game we can capture the $\G$ values of every position by only storing the values $\G(A[i,j])$. 
To show this graphically we write the value $\G(A[i,j])$ inside the box of the \Young{} for $A$, located on coordinates $(i+1)$-th row and $(j+1)$-th column (see e.g. \cref{fig:SGvalues}).

\begin{figure}[h]
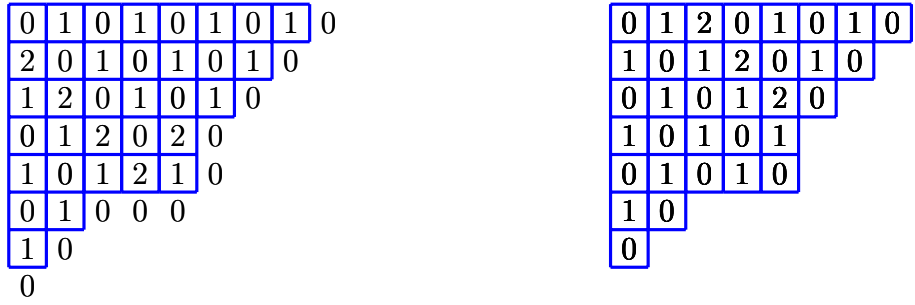

    \centering
\begin{center}
\picturepropagate  
\end{center}
    \caption{$\G$ values for subpositions in games of LCTR (left) and \King{} (right), played on $\lambda=\br{8,7,6,5^2,2,1}$. The zeros that appear outside of the Young diagram for $\L(\lambda)$ correspond to terminal positions. }
    \label{fig:SGvalues}
\end{figure}
Since each position admits  
at most two moves, we simplify the notation for the minimum excluded value and also use $\mex{a}{b}$ instead of $\Mex(\{a,b\})$.
According to \cref{eq:main_SG}, we can compute the $\G$ value of $A=\L(\lambda)$ as follows:
\begin{equation}
\G(A[i,j])=\begin{cases}
\mex{\G(A[i+1,j])}{\G(A[i,j+1])} & \text{if }A[i,j] \neq \br{}\\
0 & \text{otherwise}
\end{cases}.\label{eq:recursion}
\end{equation} 
Similarly, we can compute the $\G$ value of $A=\K(\lambda)$ as follows:
\begin{align}\label{eq:recursion2}
\G(A[i,j])=\begin{cases}
0 & \text{if } \lambda[i,j]=\br{1} \\
\Mex\{\G(A[i,j+1])\} & \text{if } \lambda[i,j]=\br{c} \text{ for }c\geq2 \\
\Mex\{\G(A[i+1,j])\} & \text{if } \lambda[i,j]=\br{1^r} \text{ for }r\geq2 \\
\mex{\G(A[i,j+1])}{\G(A[i,j+1])} & \text{otherwise}
\end{cases}.
\end{align}
See \cref{fig:MexRecursion} for an example.

We now describe $\G$ values for some families  of games introduced in \cref{sec:partitions}.

\begin{lemma}
For a positive integer $n$ we have that
$$\G(\L(\staircase_n)) = n \bmod 2\quad \text{and}\quad \G(\K(\staircase_n)) = (n-1) \bmod 2.$$
\end{lemma} 

\begin{proof}
Let $A_n=\G(\L(\staircase_n))$ or $A_n=\G(\K(\staircase_n))$.
Due to \cref{eq:recursion,eq:recursion2}, we have
\begin{align*}
    \G(A_n)&=\mex{\G(A_{n-1})}{\G(A_{n-1})}.
\end{align*}
Hence in both cases the $\SG$ values must alternate between 0 and 1 as $n$ increments by 1. 
The difference is that  $\br{}$ is a terminal position for LCTR and has $\G$ value $0$, while  $\staircase_1=\br{1}$ is a terminal position for \King{} with $\G(\K(\staircase_1))=0$. 
\end{proof}

\begin{lemma}\label{lem:ilicgama}
For positive integers $r$ and $c$ %
we have
\begin{align}
\label{eq:ilic}\G(\L(\gama_{r,c}))&=%
\begin{cases}
0 & c > 1 \mbox{ and } r > 1 \\
1 & r = 1 \mbox{ and } c \mbox{ is odd, or } c = 1 \mbox{ and } r \mbox{ is odd}  \\
2 & \text{otherwise}
\end{cases}
\end{align}
and
\begin{align}    
 \G(\K(\gama_{r,c}))&=
\begin{cases}
0 & \text{if both } c \text{ and } r \text{ are odd} \\
2 & \text{if they are of different parity and } c>1 \text{ and } r>1\\
1 & \text{otherwise} \\
\end{cases}.\label{eq:ilicking}
\end{align}

\end{lemma} 
\begin{proof}
As \cref{eq:ilic} is already shown in Ili\' c \cite{Ilic2019} we only consider the case of \King{}. 
If $c=1$ or $r=1$, then induction shows that $\G(A)$ is 0 if $cr$ is odd and 1 otherwise. 
Hence $\G(\B)$ and $cr$ are of distinct parity.
If $c>1$ and $r>1$, then by definition, 
\begin{equation*}
\G(\B)=\mex{\G(\B[1,0])}{\G(\B[0,1])}
=\begin{cases}
0 &\text{if } r\text{ and }c \text{ are odd} \\
1 &\text{if } r\text{ and }c \text{ are even} \\
2 &\text{otherwise}
\end{cases}.
\qedhere
\end{equation*}
\end{proof}

\begin{observation}\label{obs:max2}
Let $A$ be a position in any impartial game under normal play. Then $\G(A)$ does not exceed the number of allowable moves from position $A$.
\end{observation}

Due to \cref{eq:chainsubpartition},
every permissible sequence with the same number of left column and top row moves leads to the same position.

Let $G$ be any game on partitions, and suppose that for any partition $\lambda$ the subpositions reachable in one move from $\lambda$ are $\mu_1, \ldots, \mu_r$ if and only if the subpositions reachable in one move from $\lambda'$ are $\mu_1', \ldots, \mu_r'$. In this case, the game trees of $G$ on $\lambda$ and $\lambda'$ are obviously isomorphic, we say that $G$ is \emph{conjugation-invariant}, and the following lemma applies. 

\begin{lemma}\label{lem:transposeG}
Let $\lambda$ be a partition. 
Then in LCTR or \King{}, we have
\begin{enumerate}
    \item $\lambda'[i, j]=\lambda[j, i]'$ and also \label{it:transp1} 
    \item $\G(\L(\lambda'))=\G(\L(\lambda))$ and $\G(\K(\lambda'))=\G(\K(\lambda))$. \label{it:transp2}
\end{enumerate}
\end{lemma}

\section{Misère LCTR, truncation, and \King{}}\label{sec:LCTR misere}

As discussed in \cref{sec:imgames}, it is in general difficult to analyze misère play of a given combinatorial game.
Nevertheless, in this section we show that misère LCTR can be analyzed using  Sprague-Grundy theory.
In fact, \King{} is misère LCTR without its terminal positions. 
\begin{lemma}\label{lem:king=misereLCTR}
Let $\lambda\neq \br{}$ be a partition.
Then $\lambda$ corresponds to  a $\P$-position in misère LCTR if and only if it corresponds to a $\P$-position in \King{} (under normal play).
\end{lemma}
\begin{proof}

By the definition of $\P$ and $\N$ positions (\cref{def:PN-recursion}), all subpositions $\br{}$ in misère LCTR belong to $\N$ while all subpositions $\br{1}$ are in $\P$. 
For \King{}, we have that all corners  are in $\P$. Recall that empty partition $\br{}$ is not valid input partition for the game \King{}.

Now consider either of the games (misère LCTR or \King{}) played on a row $\br{c}$, $c>1$, and observe that such a game belongs to $\N$ if and only if the same game played on $(c-1)$ belongs to $\P$. 
The same arguments hold for the case of the column game, when either game is played on the partition $\br{1^r},r>1$. We conclude that the membership in $\N$ or $\P$ coincides for both games, whenever played on a column, or on a row partition.

Consider now an arbitrary partition $\lambda\neq \br{}$ and consider both games played on $\lambda$. 
We already established that in the cases of subpartitions $\br{c}$ or $\br{1^r}$, both belong to $\P$ or $\N$ in the same way.
Since 
the definition of $\P$ and $\N$  (\cref{def:PN-recursion})
computes membership of positions from both games recursively in the same way, this concludes the proof of the claim. 
\end{proof}

We now define and motivate the truncation of any impartial game. 

\begin{definition}[truncation] \label{def:truncation}
 Let  $G$ be  a finite impartial game starting on position $p$. 
 The \emph{truncation} of $G$ is an impartial game  $\norm G $ obtained from $G$ as follows:
 \begin{itemize}
    \item The starting position of $\norm G$ corresponds to $p$.
    \item $\trunc G$ has the same moves as $G$ except moves $A\to B$ where $B$ is a terminal position of $G$ no longer exist.
\end{itemize}
We will assume $\trunc G$ is played under normal play, unless stated otherwise.
\end{definition}
\cref{lem:king=misereLCTR} motivates the following theorem, from which it can be derived as a corollary.

\begin{theorem} \label{thm:truncation}

Let $G$ be an impartial misère game and let $p$ be a non-terminal position of $G$. Then $p$ is a $\P$-position of $G$ if and only if $p$ is a $\P$-position of $\trunc{G}$ under normal play. 
\end{theorem}

\begin{proof}
Let $G$ be a finite impartial misère game, let $H=\norm G$, and denote the sets of their $\P$-positions by $\P_G,\P_{H}$, respectively.
Note that the deleted $G$-terminal (i.e. terminal in $G$) positions are not $\P_G$, by definition of misère, which supports our claim. 
All the remaining positions exist in both $G$ and $H$. 
Since $G$ is a finite game, for every position $A$, all of its subpositions have smaller height than $A$.
Hence we can establish $\P_G=\P_{H}$ by induction on the height of the game tree of $G$.

First we show that the claim holds for all subpositions $A$ such that $A$ is of height $1$ in $G$. 
The only moves from $A$ are to $G$-terminal position which are all $\N_G$, thus $A\in \P_{G}$ follows.
Since $A$ is $H$-terminal and $H$ is under normal play, it holds trivially that $A\in \P_{H}$.

Now assume that the claim holds for all subpositions reachable from $A$. %
Let $\mathcal T$ denote the terminal positions to which we can move in $G$ from $A$, and let $\mathcal{B}$ correspond to the set of non-terminal positions reachable from $A$ in one move. 
Recall that $\mathcal{T} \subseteq \N_G$. 
We consider two cases depending on whether $A$ is $\P_G$ or $\N_G$. 
\begin{itemize}
    \item Let $A \in \P_G$. Then we have $\mathcal{B} \subseteq \N_G$ by 
    the definition of $\P$ and $\N$ positions (\cref{def:PN-recursion}).
    By induction hypothesis we have for all $\mathcal{B} \subseteq \N_{H}$. Since $A$ can only move to positions $\mathcal{B}$ in $H$, we have that $A \in \P_{H}$.
    
    \item Let $A \in \N_G$. Then there exists $B\in \mathcal{B}$ s.t. $B \in \P_G$ by the definition of $\P$ and $\N$. 
    By induction hypothesis we have that $B \in \P_{H}$. 
    Since position $A$ can move to position $B$ in $H$, we have that $A \in \N_{H}$.\qedhere
\end{itemize}
\end{proof}

The above theorem directly implies that 
$\N$-positions also coincide except on the deleted terminal positions of the original game, as they do not exist in the truncation.

While the functions $\SG(\trunc \cdot)$ and $\SG^-(\cdot)$ have the same support (with the exception of terminal positions), they may have different values within the support.

We find it noteworthy that  the truncation of LCTR is played on the same \Young{} under the rules of \King{}. 
Indeed, since the terminal positions of any LCTR game $G$ are “located” outside its \Young{}, removing them does not alter the diagram for $\norm G$. 

One might also consider the misère variant of \King{} in which case
its terminal positions are the corners of the corresponding \Young{}. 
Observe that for $\lambda \vdash n$, where $n \geq 2$, we have
\begin{equation} \label{eq:truncdr}
\trunc{\K(\lambda)} = \K(\lambda^-) \end{equation}
where $\lambda^-$ is obtained from $\lambda$ by removing all corners of $\lambda$. Repeated truncations of a game may give rise to a hierarchy of games. In the case of repeated truncation of \King{}, \cref{eq:truncdr} gives
$$\left \{ \text{trunc}^{i+1}(\K(\lambda)) : \lambda \vdash n, \ n \in \mathbb Z^+ \right \} \subsetneq \left\{ \text{trunc}^{i}(\K(\lambda)) : \lambda \vdash n, \ n \in \mathbb Z^+\right\}$$
for every positive integer $i$. 

\cref{lem:king=misereLCTR} further motivates the study of the $\G$ values of \King{}, which is the focus of this section.
We now show how $\G$ values propagate along diagonals, as seen in \cref{fig:SGvalues} (right).

Let $\lambda$ be a nonempty partition and let $\cor(\lambda) = \max\{i+j : (i, j) \mbox{ is a corner of } \lambda\}$. Observe that $\cor(\lambda) > \cor(\lambda[i,j])$ whenever $i > 0$ or $j > 0$. 

\begin{lemma}\label{propagatem}
Let $\lambda$ be a partition and let $i,j \geq 1$ be integers such that $\B=\K(\lambda)$ and $\B[i,j]$ are positions in \King{}. 
Then $\G(\B[i-1,j-1])=\G(\B[i,j])$.
\end{lemma} 

\begin{proof} Note that $\G(\B[i-1,j-1])=\G(\B[i,j])$ if and only if $A[i,j] + A[i-1,j-1] \in \P$. We proceed by induction on $\cor(\lambda[i,j])$. 

For the base step, suppose $\cor(\lambda[i, j])=0$. Then $A[i,j]$ is in $\P$, so $A[i-1, j]$ and $A[i, j-1]$ are in $\N$. Thus $A[i-1, j-1]$ is in $\P$ so $A[i, j] + A[i-1,j-1]$ is in $\P$. 

For the induction step, suppose that $\cor(\lambda[i,j]) > 0$. If the opponent moves in the $A[i-1,j-1]$ summand, we play in that summand again to obtain $A[i,j] + A[i,j]$ which is in $\P$. If the opponent moves in the $A[i,j]$ summand, then we mimic that move in the $A[i-1,j-1]$ summand to obtain either $A[i+1,j] + A[i,j-1]$ or $A[i,j+1] + A[i-1,j]$. The induction hypothesis applies because $\cor(\lambda[i+1, j])$  and $\cor(\lambda[i, j+1])$ are both less than $\cor(\lambda[i,j])$. Therefore, $A[i+1,j] + A[i,j-1]$ and $A[i,j+1] + A[i-1,j]$ are in $\P$. 
\end{proof}

\noindent
Due to \cref{propagatem}, we can now state the main theorem of this section.

\begin{theorem}\label{thm:propking}
Let $\lambda$ be a partition, let $\B=\K(\lambda)$ be a game of \King{} and let $\ell=\D(\lambda)$. Then 
$$\G(\B)=\G(\B[1,1])=\G(\B[2,2])=\dots=\G(\B[\ell-1,\ell-1]).$$
\end{theorem}

To put things in perspective, to calculate the $\G$ value of a game of \King{}, by \cref{thm:propking} we can calculate the subpartition at the end of the diagonal, which turns out to be always be a gamma partition, which we know how to calculate by \cref{lem:ilicgama}.

\section{LCTR normal play}
\label{sec:LCTR}

In this section we focus on LCTR as described in \cref{sec:lctr}.
We start by showing how to obtain $\SG$ values of small games, namely for LCTR games of Durfee length at most three.
We postpone the analysis of computational complexity until \cref{sec:computing}.
\subsection{LCTR games of Durfee length at most two}

Let $\lambda$ be a partition and let $A=\L(\lambda)$. 
If $\D(\lambda)=0$, then  $\lambda=\br{}$ with $\SG(A)=0$.
The case when $\D(\lambda)=1$ is handled by \cref{lem:ilicgama}; see \cite{Ilic2019}.

If $\D(\lambda)=2$, then we can write
$\lambda=\br{\lambda_1,\lambda_2,2^a,1^b}$ where $\lambda_1\geq \lambda_2 \geq 2$, for some non-negative integers $a,b$.
We resolve $\SG(A)$ in two phases (also see \cref{fig:cut}):
\begin{enumerate}
    \item First we obtain values $\SG(A[2,0]),\SG(A[2,1]),\SG(A[0,2]),\SG(A[1,2])$.\label{it:two-row-games}
    \item Given those boundary values we resolve the remaining $\SG$ values, namely $\SG(A[1,0])$, $\SG(A[0,1])$, $\SG(A[1,1])$, and finally $\SG(A[0,0])=\SG(A)$.\label{it:remaining-4}
\end{enumerate}

To resolve the $\SG$ values mentioned in \cref{it:two-row-games}, we use the following lemma from \cite{Ilic2019}.

\begin{lemma}[one-row game]\label{lem:1rowgame}
Let $A = \L(\lambda)$ with $\lambda=\br{\lambda_1}$. 
Then 
$$\G(A)=\begin{cases}
0 & \text{if $\lambda_1=0$} \\
1 & \text{if $\lambda_1$ is odd} \\
2 & \text{otherwise}
\end{cases}.$$
\end{lemma}

Clearly one may obtain $\SG(A[2,1])$ and $\SG(A[1,2])$ by using \cref{lem:1rowgame} together with input partition $\br{a}$, or $\br{\lambda_2-2}$, respectively.
In the former case we used the fact that $\SG(\L(\br{a}))=\SG(\L(\br{1^a}))$, due to \cref{lem:transposeG}.

We write $\bcancel{b}$ for an arbitrary integer from the set $\{0,1,2\}\setminus \{b\}$. 
This allows us to simplify some computations such as $\mex{\bcancel{0}}{\bcancel{0}}=0$, or $\mex{0}{\bcancel{1}}=1$.

\begin{lemma}[two-row game]\label{lem:2rowgame}
Let $A = \L(\lambda)$ with  $\lambda=\br{\lambda_1,\lambda_2}$.
\begin{enumerate}
        \item \label{lem:2rowgame1}  If $\lambda_1=\lambda_2$ then $\G(A)=\begin{cases}
0 & \text{if } \lambda_1 \text{ is even} \\
2 & \text{if } \lambda_1 \text{ is odd}
\end{cases}.$\label{it:two-equal-rows}
        \item If $\lambda_1>\lambda_2>0$ then $\G(A)=\begin{cases}
0 & \text{if } \lambda_2 \text{ is odd} \\
1 & \text{if } \lambda_2 \text{ is even}
\end{cases}.$ \label{it:two-diff-rows}
    \end{enumerate}
\end{lemma}

\begin{proof}
Let $0\leq i \leq \lambda_2$.
Throughout this proof keep in mind that
$\G(A[1,i])=\G(\L(\br{i}))$ can be obtained due to \cref{lem:1rowgame}.

Now recall $\G(A[0,\lambda_1])=0$ and consider the case $\lambda_1=\lambda_2$. 
It thus follows that $\G(A[0,\lambda_2-1])=\mex{1}{0}=2$, and similarly $\G(A[0,\lambda_2-2])=\mex{2}{2}=0$.
In particular, those two $\G$ values alternate through subgames $A[0,j]$ where $0\le j<\lambda_2$, yielding the desired formula.

Finally consider the case $\lambda_1>\lambda_2>0$.
It follows from \cref{lem:1rowgame} that $\G(A[0,\lambda_2])=\bcancel{0}$ and thus $\G(A[0,\lambda_2-1])=\mex{1}{\bcancel{0}}=0$, and similarly $\G(A[0,\lambda_2-2])=\mex{2}{0}=1$.
In particular, those two $\G$ values alternate through subgames $A[0,j]$ where $0\le j<\lambda_2$, yielding the desired formula.
\end{proof}

\cref{lem:1rowgame,lem:2rowgame} can now be used to obtain the remaining values $\SG(A[2,0])$ and $\SG(A[0,2])$,   together with input partition $\br{a+b,b}$, or $\br{\lambda_1-2,\lambda_2-2}$, respectively.
In the former case we again used \cref{lem:transposeG} to obtain  $\SG(\L(\br{a+b,b}))=\SG(\L(\br{2^a,1^b}))$.

We now resolve the remaining $\SG$ values as announced in \cref{it:remaining-4}.
One may now use \cref{eq:recursion} to express those $\G$ values as 
\begin{align}
\SG(A[1,1])&=\mex{\SG(A[2,1])}{\SG(A[1,2])} & \SG(A[0,1])&=\mex{\SG(A[0,2])}{\SG(A[1,1])} \nonumber \\
\SG(A[1,0])&=\mex{\SG(A[2,0])}{\SG(A[1,1])} & \SG(A)&=\mex{\SG(A[0,1])}{\SG(A[1,0])}
\label{eq:remaining-4}
\end{align}

For simplicity, call the $\G$ values of the subgames $A[2,0],A[2,1]$ by $\alpha_{20},\alpha_{21}$ and the subgames $A[0,2],A[1,2]$ by $\alpha_{02},\alpha_{12}$, respectively. 
By definition, we may express the $\G$ value of game $A$ as 
\begin{align}
\SG(A)&=\mex{\mexx{\mexx{\alpha_{21}}{\alpha_{12}}}{\alpha_{02}}}{\mexx{\alpha_{20}}{\mexx{\alpha_{21}}{\alpha_{12}}}}.\label{eq:2x2game}
\end{align}

\begin{figure}
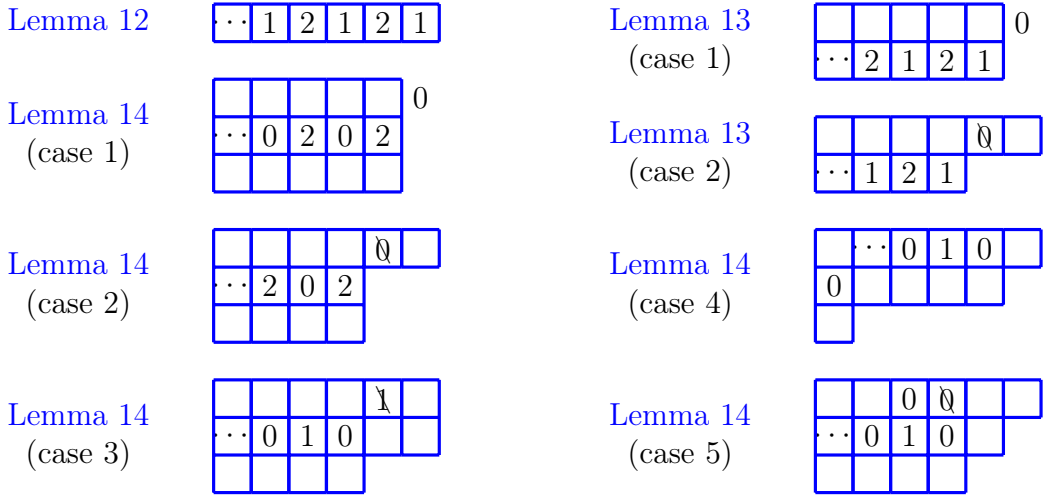

    \centering
    \picturecolumns
    \caption{Cases from \cref{lem:1rowgame,lem:2rowgame,lem:3rowgame}. The symbol $\cdots$ denotes alternating values from the two boxes  right from it.}
    \label{fig:columns}
\end{figure}

\subsection{LCTR games of Durfee length three}
Let $\lambda$ be a partition such that $\D(\lambda)=3$, and let $A=\L(\lambda)$. 
We  write
$\lambda=\br{\lambda_1,\lambda_2,\lambda_3,3^a,2^b,1^c}$ where $\lambda_1\geq \lambda_2 \geq \lambda_3 \geq 3$.
We again resolve $\SG(A)$ in the following two phases:
\begin{enumerate}
    \item  First we obtain values $\SG(A[3,0])$, $\SG(A[3,1])$, $\SG(A[3,2])$, $\SG(A[0,3])$, $\SG(A[1,3])$, $\SG(A[2,3])$.\label{it:three-row-games}
    \item Given the values from \cref{it:three-row-games} we resolve the remaining nine $\SG$ values, namely 
    \begin{align}
        &\SG(A[2,2]), &&\SG(A[1,2]), &&\SG(A[0,2]), \nonumber\\
        &\SG(A[2,1]), &&\SG(A[1,1]), &&\SG(A[0,1]), \label{it:remaining-9}\\
        &\SG(A[2,0]), &&\SG(A[1,0]), &&\SG(A[0,0])=\SG(A).\nonumber
    \end{align}
  
\end{enumerate}

To resolve the $\SG$ values mentioned in \cref{it:three-row-games} notice that values $\SG(A[3,1])$, $\SG(A[3,2])$, $\SG(A[1,3])$, and $\SG(A[2,3])$ may be obtained by using \cref{lem:1rowgame,lem:2rowgame}, together with partitions $\br{b+c,c},$ $\br{c},$ $\br{\lambda_2-3,\lambda_3-3}$, and $\br{\lambda_3-3}$, respectively.

To obtain $\SG(A[3,0])$ and $\SG(A[0,3])$ we introduce the following lemma.

\begin{lemma}[three-row game]\label{lem:3rowgame}
Let $A = \L(\lambda)$ with  $\lambda=\br{\lambda_1,\lambda_2,\lambda_3}$.
\begin{enumerate}
    \item \label{lem:3rowgame1} If $\lambda_1=\lambda_2=\lambda_3>0$ then
        $\G(A)=\begin{cases}
        0 & \text{if } \lambda_3\geq 3 \text{ and odd} \\
        1 & \text{if } \lambda_3=1 \text{ or }( \lambda_3\geq 4 \text{ and even})  \\
        2 & \text{if } \lambda_3=2 \\
        \end{cases}.$
        \item If $\lambda_1>\lambda_2=\lambda_3>0$ then 
        $\G(A)=\begin{cases}
        0 & \lambda_3 \text{ is odd}\\
        1 & \lambda_3 \text{ is even}\\
        \end{cases}.$
        \item If $\lambda_1=\lambda_2>\lambda_3>0$ then $\G(A)=
        \begin{cases}
        0 & \lambda_3 \text{ is even}\\
        1 & \lambda_3 \text{ is odd}\\
        \end{cases}.$
        \item if $\lambda_1>\lambda_2>\lambda_3=1$ then $\G(A)=
        \begin{cases}
        1 & \lambda_2 \text{ is even}\\
        2 & \lambda_2 \text{ is odd}\\
        \end{cases}.$
        \item if $\lambda_1>\lambda_2>\lambda_3>1$ then $\G(A)=
        \begin{cases}
        0 & \lambda_3 \text{ is even}\\
        1 & \lambda_3 \text{ is odd}\\
        \end{cases}.$
    \end{enumerate}

\end{lemma}

\begin{proof}
We resolve the above-mentioned five cases separately, relying heavily on \cref{lem:2rowgame}.
Recall from \cref{lem:2rowgame} that the $\G$ values of positions  $A[1,j]$ with $0\le j<\lambda_3$ alternate between $0$ and $2$ whenever $\lambda_2=\lambda_3$, or between $1$ and $2$ otherwise. 
Also observe that the value $\G(A[0,\lambda_3])$ can be obtained from either \cref{lem:1rowgame} or \cref{lem:2rowgame}.
\begin{enumerate}
    \item ($\lambda_1=\lambda_2=\lambda_3>0$):
    We have $\G(A[0,\lambda_3])=0$, while the $\G$ values of subgames $A[1,j]$ with $0\le j<\lambda_3$ alternate between $2$ and $0$, starting with $\G(A[1,\lambda_3-1])=2$.
    
    \item ($\lambda_1>\lambda_2=\lambda_3>0$):
    We have $\G(A[0,\lambda_3])=\bcancel{0}$ 
    while the $\G$ values of subgames $A[1,j]$ with $0\le j<\lambda_3$ alternate between $2$ and $0$, starting with $\G(A[1,\lambda_3-1])=2$.
    \item ($\lambda_1=\lambda_2>\lambda_3>0$):
    We have $\G(A[0,\lambda_3])=\bcancel{1}$, while the $\G$ values of subgames $A[1,j]$ with $0\le j<\lambda_3$ alternate between $1$ and $0$, starting with $\G(A[1,\lambda_3-1])=0$.
    \item ($\lambda_1>\lambda_2>\lambda_3=1$):
    We have $\G(A[0,1])=\G(\L(\br{\lambda_1-1,\lambda_2-1}))=\lambda_2 \bmod 2 $, while the $\G$ value of subgame $\G(A[1,0])=0$.
    \item ($\lambda_1>\lambda_2>\lambda_3>1$):
    Due to \cref{it:two-diff-rows} of \cref{lem:2rowgame} and
    \cref{cor:adjacentMex} we have $\G(A[1,\lambda_3-1])=0$, $\G(A[1,\lambda_3-2])=1$ and also $\G(A[0,\lambda_3-1])=\bcancel{0}$.
    This gives $\G(A[0,\lambda_3-2])=0$.
    The $\G$ values of subgames $A[1,j]$ with $0\le j\le\lambda_3-2$ alternate between $1$ and $0$, starting with $\G(A[0,\lambda_3-2])=0$.
\end{enumerate}
From the $\G$ values of subgames $A[1,j]$ with $0\le j\le \lambda_3-1$, we derive $\G$ values of subgames $A[0,j]$ with $0\le j\leq \lambda_3-1$ and obtain the desired result for each case.
\end{proof}

\cref{lem:3rowgame,lem:2rowgame,lem:1rowgame}, together with partitions $\br{\lambda_1-3,\lambda_2-3,\lambda_3-3}$ and $\br{a+b+c,b+c,c}$ resolve
$\SG(A[0,3])$ and $\SG(A[3,0])$, respectively.
The remaining nine Sprage-Grundy values mentioned in \cref{it:remaining-9} can, similarly as in \cref{eq:remaining-4}, be expressed as
\begin{align*}
\SG(A[2,2])&=\mex{\SG(A[2,3])}{\SG(A[3,2])}, & 
\SG(A[2,1])&=\mex{\SG(A[2,2])}{\SG(A[3,1])},  \\
\SG(A[2,0])&=\mex{\SG(A[2,1])}{\SG(A[3,0])}, & 
\SG(A[1,2])&=\mex{\SG(A[1,3])}{\SG(A[2,2])}, \\ 
\SG(A[1,1])&=\mex{\SG(A[1,2])}{\SG(A[2,1])}, &
\SG(A[1,0])&=\mex{\SG(A[1,1])}{\SG(A[2,0])}, \\ 
\SG(A[0,2])&=\mex{\SG(A[0,3])}{\SG(A[1,2])}, &
\SG(A[0,1])&=\mex{\SG(A[0,2])}{\SG(A[1,1])},  
\end{align*}
and finally
\begin{equation}
\SG(A)=\mex{\SG(A[0,1])}{\SG(A[1,0])}. \label{eq:remaining-9}
\end{equation}

\begin{figure}[b]
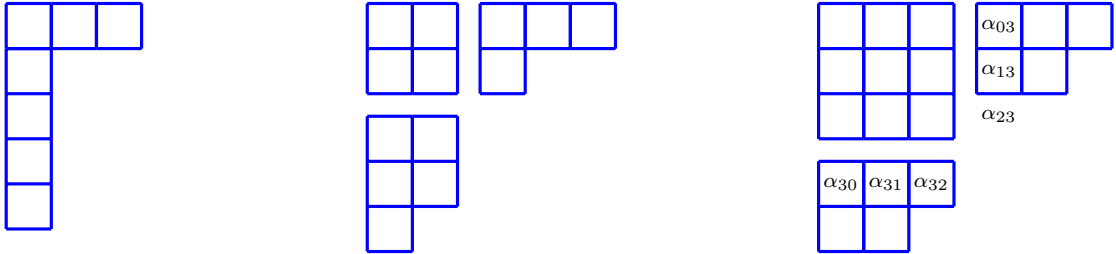

\centering
\picturecut
\caption{Resolving games played on $\br{3,1^4}$, $\br{5,3,2^2,1}$, and $\br{6,5,3^2,2}$. %
}\label{fig:cut}
\end{figure}

For simplicity, call the $\G$ values of the subgames $A[3,0],A[3,1],A[3,2]$ by $\alpha_{30},\alpha_{31},\alpha_{32}$ and the subgames $A[0,3],A[1,3],A[2,3]$ by $\alpha_{03},\alpha_{13},\alpha_{23}$, respectively. 
As in \cref{eq:2x2game}, we
express the $\G$ value of game $A$ as 
\begin{align}
{\mexx{\mexx{\alpha_{30}}{\mexx{\alpha_{31}}{\mexx{\alpha_{32}}{\alpha_{23}}}}}{\mexx{\mexx{\alpha_{31}}{\mexx{\alpha_{32}}{\alpha_{23}}}}{\mexx{\mexx{\alpha_{32}}{\alpha_{23}}}{\alpha_{13}}}}}\phantom{.} \quad \mexsymb \nonumber \\
{\mexx{\mexx{\mexx{\alpha_{31}}{\mexx{\alpha_{32}}{\alpha_{23}}}}{\mexx{\mexx{\alpha_{32}}{\alpha_{23}}}{\alpha_{13}}}}{\mexx{\mexx{\mexx{\alpha_{32}}{\alpha_{23}}}{\alpha_{13}}}{\alpha_{03}}}}. \quad \phantom{\star}\label{eq:3x3game}
\end{align}

\subsection{LCTR games of Durfee length at least four}
We now give a bit more complicated version of the propagation from \cref{propagatem}, which holds for LCTR normal play.

\begin{lemma}\label{propagate}
Let $\lambda$ be a partition, let $A=\L(\lambda)$ be a game of LCTR, and let $i,j\ge1$ be such that $\lambda[i,j]\neq \br{}$.
\begin{enumerate}
    \item If $\G(A[i,j])=0$, then $\D(\lambda[i,j])\geq 1$ and
    $\G(A[i-1,j-1])=0$. \label{propagate0}
    \item If $\G(A[i,j])=1$ and $\D(\lambda[i,j])\geq 2$, then  $\G(A[i-1,j-1])=1$. \label{propagate1}
    \item If $\G(A[i,j])=2$ and $\D(\lambda[i,j])\geq 3$, then  $\G(A[i-1,j-1])=2$. \label{propagate2}
\end{enumerate}
\end{lemma}

\begin{proof}
We have that  $\lambda[i',j']\neq \br{}$ for any choice of $0\le i'<\D(\lambda)$ and $0\le j'<\D(\lambda)$. In particular, for any corresponding subgame $A[i',j']$ we may use the first case of recursion in \cref{eq:recursion}.
We consider three cases depending on the value of  $\G(\B[i,j])$.

\begin{enumerate}
\item Suppose $\G(A[i,j])= 0$. 
    Due to \cref{cor:adjacentMex} we have
    \begin{align*}
        \G(A[i-1,j])\neq 0 \neq 
        \G(A[i,j-1]).
    \end{align*}
    Therefore  $\G(A[i-1,j-1])=\mex{A[i-1,j]}{A[i,j-1]}=\mex{\bcancel{0}}{\bcancel{0}}=0$.
    Since  $\lambda[i,j]\neq \br{}$, clearly $\D(\lambda[i,j])\geq 1$.
\item Suppose $\G(A[i,j])= 1$ and $\D(\lambda[i,j])\geq 2$, in particular neither $\lambda[i,j+1]$ nor $\lambda[i+1,j]$ is empty. 
    First notice that by \cref{cor:adjacentMex} we have
    \begin{align*}
        \G(A[i,j-1])&\neq  1 \neq \G(A[i-1,j]).
    \end{align*} 
    This implies $\G(A[i-1,j-1])=\mex{\bcancel{1}}{\bcancel{1}}\leq 1$.
 Now assume $\G(A[i-1,j-1])=0$. 
 From \cref{cor:adjacentMex,obs:max2} it follows that
 \begin{align*}
    &\G(A[i,j-1])=\G(A[i-1,j])=2, \text{ and}\\ &\G(A[i+1,j-1])=\G(A[i-1,j+1])=0.
  \end{align*}

 Since neither $\lambda[i,j+1]$ nor $\lambda[i+1,j]$ is empty, 
 by \cref{cor:adjacentMex} we have 
 \[\G(A[i+1,j]) = \G(A[i,j+1])=2.\]
This leads to $\G(A[i,j])=\mex{2}{2}\neq 1,$ a contradiction.
 \item Suppose $\G(A[i,j])= 2$ and  $\D(\lambda[i,j])\geq 3$%
 .
Due to  $\mex{2}{0}=1$ it follows that
$$
\{\G(A[i,j+1]),\G(A[i+1,j])\}=\{0,1\}.
$$
Since $\D(A[i,j])\geq 3$ we have that $\D(A[i,j+1])\geq 2$ and $\D(A[i+1,j])\geq 2$.
In addition, \cref{propagate0,propagate1} imply that
$$
\{\G(A[i-1,j]),\G(A[i,j-1])\}=\{0,1\}.
$$
Hence $\G(A[i-1,j-1])=\mex{0}{1}=2$. \qedhere
\end{enumerate}
\end{proof}

\noindent
The following corollary will be used in \cref{sec:computing}, for efficient computation of $\G$ values.

\begin{figure}[]
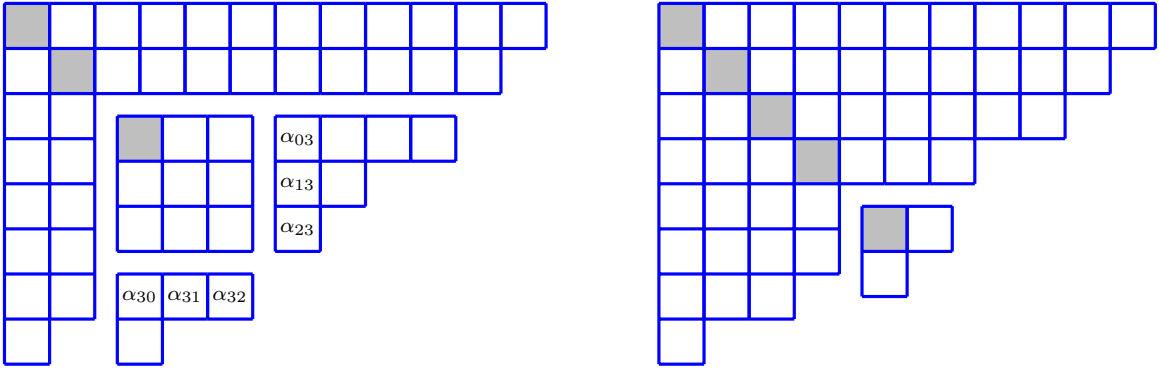

\centering
\picturebigcut
\caption{Resolving the games LCTR (left) and \King{} (right), played on $\lambda = \br{12,11,9,7,6,5,3,1}$. The gray boxes are filled with the same $\G$ values.}\label{fig:bigcut}
\end{figure}

\begin{corollary}\label{cor:propLCTR}
Let $\lambda$ be a partition, let $\ell = \D(\lambda)-3\geq 0$, and let $A= \L(\lambda)$. 
Then 
\[ \G(A)=\G(A[\ell,\ell]).\]
\end{corollary}

\begin{proof}
Let $\bar A=A[\ell,\ell]$.
Since $\ell \geq 0$ it follows that the position $\bar A$ is properly defined and we have $\D(\bar A)=3$. By \cref{propagate,obs:max2}, regardless of the value of $\G(\bar A)$ we  infer 
\begin{equation*}
\G(\bar A)=\G(A[\ell-1,\ell-1])=\cdots=\G(A[\ell-\ell,\ell-\ell])=\G(A). \qedhere
\end{equation*}
\end{proof}

In view of this corollary we restrict ourselves to the study of games with Durfee length at most 3.
As an application of this corollary we now compute $\G(\br{c^r})$.

\begin{lemma} \label{SGbox}
For positive integers $c$ and $r$, we have
$$ \G(\L(\rectangle_{r,c})) =
\begin{cases}
0 & \text{if } c > 1 \mbox{ and } r > 1 \mbox{ and } c+r \mbox{ is even} \\
2 & \text{if } c \leq 2 \mbox{ or } r \leq 2, \mbox{ and } c + r \mbox{ is odd}\\
1 & \text{otherwise}
\end{cases}.$$ 
\end{lemma}

\begin{proof}
By \cref{lem:transposeG} we can assume $c \geq r$, otherwise we can take its conjugate which has this property.
If $r = 1$, then the formula holds by \cref{lem:1rowgame}. If $r = 2$, then the result holds by \cref{lem:2rowgame1} of \cref{lem:2rowgame}. 
If $r\geq 3$ then by \cref{cor:propLCTR} the game $\G(\L(\rectangle_{c-r+3,3}))=\G(\L(\rectangle_{c,r}))$ and this game can be solved with \cref{lem:3rowgame1} of \cref{lem:3rowgame}.
\end{proof}

\section{The hierarchical classification of LCTR and \King{}\label{sec:hierarchy}}
In this section we classify the newly introduced games LCTR and \King{}, in terms of the hierarchy presented by Gurvich and Ho~\cite{GURVICH201854}.

\begin{lemma}\label{lem:zerotwo}
LCTR and \King{} do not admit $(0,2)$- and $(2,0)$-positions.
\end{lemma}
\begin{proof}
    We prove the statement for both games at the same time, since all the arguments are the same. 
        Consider a minimal counterexample $\lambda$. 
    We show that it is neither a $(0,2)$-, nor a $(2,0)$-position.
\begin{description}
    \item[Suppose  $\lambda$ is a $(2,0)$-position.]
    Notice that both moves are available because no $(2,0)$-position can be obtained by less then two moves. 
    Without loss of generality let $\lambda[1,0]$ be a $(0,\bcancel{0})$-position and let $\lambda[0,1]$ be a $(1,\bcancel{0})$-position. 
Due to minimality $\lambda[1,0]$ must be a $(0,1)$-position.

Notice that the move $\lambda[0,1] \to \lambda[1,1]$ is legal in both games. 
Otherwise, $\lambda[0,1] \to \lambda[0,2]$ is legal, and hence $\lambda[0,1]$ is a $(1,1)$-position. 
This implies that 
for $i>1$, $\lambda[0,i]$ alternates between $(0,0)$- and $(1,1)$-positions, contradicting the finiteness of $\lambda$.

Due to minimality, observe that  $\lambda[1,1]$ is a $(2,2)$-position, which implies that $\lambda[0,1]$ is a $(1,1)$-position.
Next, notice that $\lambda[0,2]$ is a $(0,0)$-position, 
due to $(1,1)$ being obtained as a mex of a $(2,2)$-position and this position.
But then we clearly have that $\lambda[1,2]$ is a $(1,1)$-position, because it is adjacent to both $(0,0)$- and $(2,2)$-positions. 
By extension, $\lambda[2,1]$ is a $(0,0)$-position,
due to $(2,2)$ being obtained as a mex of a $(1,1)$-position and this position.

Focusing on $\lambda[2,0]$, first notice that it must be a $(\bcancel{0},0)$-position, which by minimality implies that $\lambda[2,0]$ is a $(1,0)$-position.
This is a contradiction as its neighbour, $\lambda[2,1]$, is a $(0,0)$-position.
    \item[Suppose  $\lambda$ is a $(0,2)$-position.]
   For this case the proof is exactly the same, with swapped coordinates of all Sprague-Grundy pairs. 
\end{description}
\end{proof}

\begin{theorem}
LCTR and \King{} are domestic and  returnable, but neither tame nor forced. 
\end{theorem}
\begin{proof}
To see that the games are returnable,
suppose by contradiction there is a non-terminal $(0,1)$- position  $\lambda$ which does not lead to a $(1,0)$-position. 
But then it leads to a $(2,0)$ position, contradicting \cref{lem:zerotwo}. 
A symmetric argument shows that every $(1,0)$-position can move to a $(0,1)$-position. Thus it follows that the games are returnable.

To see that the games are domestic, first note that both $\SG$ and $\SG^-$ values cannot exceed $2$, thus 
it is enough to show $(0,2)$ and $(2,0)$-positions are impossible, which \cref{lem:zerotwo} guarantees.

They are not tame due to existence of $(2,1)$-positions. They are not forced by existence of moves from a $(1,0)$-position to a $(2,2)$-position. See \cref{fig:sgpairs}. 
\end{proof}

\begin{figure}[!h]
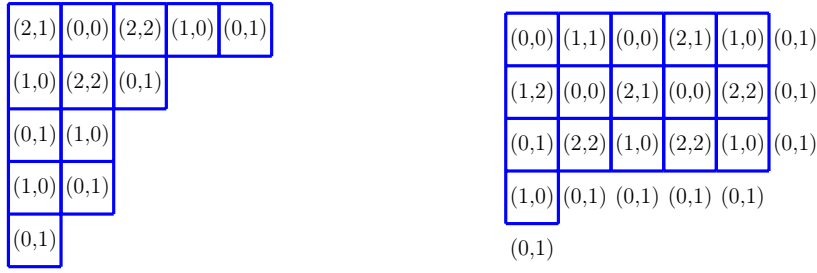

    \centering
    \picsgandmiseredr \hspace{6em}
    \picsgandmiserelctr
    \caption{SG pairs, \King{} on the left, LCTR on the right.}
    \label{fig:sgpairs}
\end{figure}

\section{The complexity of LCTR and \King{}\label{sec:computing}}

In this section we discuss computational aspects of both LCTR and \King{}. 
Let $n=\sum_{i=1}^r\lambda_i$, where $\lambda=\br{\lambda_1,\dots,\lambda_r}$ is the input partition for either game. 
Complexity results will be expressed in terms of $\lambda_1$, $r$, and $n$. In \cref{sec:timecomp}, we discuss the time complexity required to compute $\G$-values of LCTR and \King{}. In \cref{sec:treecomp}, we study the number of nodes and leaves in their game trees. In \cref{sec:statespace}, we examine the state-space complexity of these games. 

\subsection{Time complexity of computing \texorpdfstring{$\G$}{SG} values}\label{sec:timecomp}

The results of this section will be expressed in terms of the number of time units needed. We consider the \emph{time unit} to be any of the following: 
\begin{enumerate}
    \item \textbf{Reading an integer.} Note that our input partition  $\lambda=\br{\lambda_1,\dots,\lambda_r}$ consists of a sequence of $r$ integers, so reading the input takes $r$ time units. 
    \item \textbf{Basic arithmetic operations.} Here we mean integer multiplication or division by $2$, 
    addition or subtraction of the given integers, or computing the parity. 
    \item \textbf{Minimum excluded value:} We consider computing a minimal excluded value to be a basic operation which takes one time unit. Note that mex computations will only be performed with integers from the set $\{0,1,2\}$.
\end{enumerate}

For a given partition $\lambda=\br{\lambda_1,\lambda_2,\dots,\lambda_r}$ we show that the $\G$ value of both LCTR and \King{}, played on $\lambda$, can be determined  $O(\log(r)) \leq O(\log(n))$ time units. 
For comparison, the dynamic programming approach by \cite{Ilic2019} required $\Theta(n)$ time units for the same task.

\begin{lemma}\label{lem:durfee}
Given a partition $\lambda=\br{\lambda_1,\ldots,\lambda_r}\vdash n$ we can find its Durfee length $\D(\lambda)$ in $O(\log(r))\le O(\log(n))$ time units.
\end{lemma}
\begin{proof}
Recall that the Durfee length is the largest $i\in \{1,\ldots,r\}$ such that $\lambda_i\geq i$.
We can find such a value in $\log(r)$ time units by using binary search.
\end{proof}

\begin{lemma} \label{lem:conj}
Given a partition $\lambda=\br{\lambda_1,\ldots,\lambda_r}\vdash n$ we can compute its conjugate $\lambda'$ in 
$O(\lambda_1\log(r))\leq O(n\log(n))$
time units.
\end{lemma}

\begin{proof}
For every $i\in\{1,\ldots,\lambda_1\}$ we 
find $\lambda'_i$
via a binary search on the list $\br{\lambda_1,\ldots,\lambda_r}$ which takes $O(\log(r))$ time units each. 
Hence, in total we need $O(\lambda_1\log(r))$ time units.
\end{proof}

 \ifthenelse{\arx=1}{}{
The following result follows easily from \cref{thm:propking} and \cref{propagate}, which describe how $\G$-values propagate in the Young diagram of a partition. 
}

\begin{theorem}
Let $\lambda=\br{\lambda_1,\lambda_2,\dots,\lambda_r}$ be a partition of $n$. 
The $\G$ values of $\K(\lambda)$ and $\L(\lambda)$ can both be computed in $O(\log(r)) \leq O(\log(n))$ time units.
\end{theorem}

 \ifthenelse{\arx=1}{
\begin{proof}
We can compute the Durfee length $\D(\lambda)$ in $O(\log(r))$ time units due to \cref{lem:durfee}. We prove the statement for both games separately.

\paragraph{Proof for \King{}}
Let $\bar \lambda=\br{\bar \lambda_1, \bar \lambda_2, \ldots, \bar \lambda_{\bar r}}$ be a partition corresponding to the smallest “diagonal subgame,” i.e. $\B[\D(\lambda)-1,\D(\lambda)-1]=\K(\br{\bar \lambda_1, \bar \lambda_2, \ldots, \bar \lambda_{\bar r}})$.
Due to \cref{thm:propking} it is enough to compute $\G$ value of $\K(\bar \lambda)$, as it corresponds of that of $K(\lambda)$. 

To this end observe that, by construction, $\bar \lambda$ is of  shape $\gama_{a,b}$. 
Here $a=\lambda_{\D(\lambda)}-\D(\lambda)+1$, while integer $b$ can be computed via binary search in $O(\log r)$ time units.\footnote{Note that we never need to explicitly compute $\bar \lambda$, as this would cause the computational complexity to grow.}
It remains to verify, by \cref{lem:ilicgama}, that the $\G$ value of $\gama_{a,b}$ can be computed in $O(1)$ time units.
Summarizing the time complexity, the computation of $\G$ value of the game of \King{} requires $O(\log(r))$ time units. 

\paragraph{Proof for LCTR}
Let $A=\L(\lambda)$.
Due to \cref{lem:durfee}, we can compute the Durfee length $\D(\lambda)$ in $O(\log(r))$ time units. We distinguish three cases, depending on the Durfee length of $\lambda$:
\begin{itemize}
    \item \textbf{Case $\D(\lambda)=1$.} Notice that $\lambda$ is of gamma shape, and the corresponding $\G$ value can be computed in $O(1)$ time units due to \cref{lem:ilicgama}.
    \item \textbf{Case $\D(\lambda)=2$.} 
    Denote the lengths of the first two columns by $r$ and $j$, i.e. $j$ is the maximal integer such that $\lambda_j\ge2$.  
    We need $O(\log(r))$ time units to compute $j$ via binary search.
    
    We express $\G(A)$ as a function of $\alpha_{02}$, $\alpha_{20}$, $\alpha_{12}$, $\alpha_{21}$, where $\alpha_{ij}=\G(A[i,j])$. 
    We calculate those four values as a function of $\lambda_1,\lambda_2,r$ and $j$, by using \cref{eq:2x2game} where 
    \begin{align*}
        \alpha_{02}&=\G(\L(\br{\lambda_1-2,\lambda_2-2})) & \alpha_{20}&=\G(\L(\br{r-2,j-2}))\\
        \alpha_{12}&=\G(\L(\br{\lambda_2-2}))       & \alpha_{21}&=\G(\L(\br{j-2})).
    \end{align*}
    The computation of $\alpha_{02},\alpha_{20},\alpha_{12},\alpha_{21}$ can be done in $O(1)$ time units; see \cref{lem:1rowgame,lem:2rowgame}.
    Afterwards the application of \cref{eq:remaining-4} can be performed in constantly many time units, so the overall computation time is $O(\log(r))$ time units. 
    \item \textbf{Case $\D(\lambda)\ge 3$.} Finally, let $\bar \lambda=\br{\bar \lambda_1, \bar \lambda_2, \ldots, \bar \lambda_{\bar r}}$ be a partition so that $\bar \lambda=\lambda[\D(\lambda)-3,\D(\lambda)-3]$.
    We do not compute $\bar \lambda$ but whenever we check a value $\bar \lambda_\ell$ we instead use $\bar \lambda_\ell=\lambda_{\D(\lambda)+\ell-3}-\D(\lambda)+3$. 

    It is enough to compute $\G$ value of $\L(\bar \lambda)$ as it corresponds of that of $\L(\lambda)$, due to \cref{cor:propLCTR}. 

Let $\bar r, j$ and $i$ be the (possibly zero) lengths of the first three columns of $\bar \lambda$, respectively.
Formally,
let $\bar r,j$ and $i$ be maximum integers such that
$\bar \lambda_{\bar r} \ge 1$, and
$\bar \lambda_j \ge 2$, and 
$\bar \lambda_i \ge 3$,
 respectively.
We need $O(\log(r))$ time units to compute $\bar r,j$ and $i$.
The value of $\G(A)$ can now be computed in $O(1)$ time units by using \cref{eq:3x3game}, where 
\begin{align*}
\alpha_{03} &=\G(\L(\br{\bar \lambda_1-3, \bar \lambda_2-3, \bar \lambda_3-3})) & \alpha_{30}&=\G(\L(\br{\bar r-3,j-3,i-3}))\\
\alpha_{13} &=\G(\L(\br{\bar \lambda_2-3, \bar \lambda_3-3})) & \alpha_{31}&=\G(\L(\br{j-3,i-3}))\\
\alpha_{23} &=\G(\L(\br{\bar \lambda_3-3})) & \alpha_{32}&=\G(\L(\br{i-3})).\qedhere
\end{align*}
\end{itemize}
\end{proof}
}{}

\subsection{The game tree of LCTR and \King{}}\label{sec:treecomp}

In this section we describe the number of nodes and leaves that the game tree of LCTR, or \King{}, might have.
 \ifthenelse{\arx=1}{
\paragraph{The number of nodes in the game tree}
Recall that $\nodes(A)$, and $\leaf(A)$, denote the number of nodes, and leaves, in the game tree of the game $A$, respectively.
The following corollary connects those notions, and follows immediately from \cref{def:truncation}.
\begin{corollary}\label{cor:leaf+node}
Let $\lambda$ be a partition. Then $\nodes(\L(\lambda))=\nodes(\K(\lambda))+\leaf(\L(\lambda))$.
\end{corollary}
Given a game $A$ of either type, recall that a position may occur several times within the game tree of $A$. 
More precisely, if $A[i,j]\neq \br{}$ then the pair $(i,j)$ gives rise to exactly $\binom{i+j}{i}$ nodes within the game tree of $A$. 

The next result follows by induction and an application of inclusion-exclusion.

\begin{lemma} For positive integers $r,c$ we have  $\nodes(\K(\rectangle_{r,c}))=\binom{r+c}{r}-1$.
\end{lemma}

\begin{proof}
First observe that if $r=1$ or $c=1$, then $\rectangle_{r,c}=\gama_{r,c}$, thus the claim holds (also see \cref{tab:tree-parm}).  
The proof is by induction on $r+c$.
Since the base case was already covered, 
assume that the claim holds for all partitions $\rectangle_{r',c'}$ with $r'+c'<\ell$ and let $r+c=\ell$.
By the principle of inclusion-exclusion first observe the following recurrence of  $\nodes(\K(\rectangle_{r,c}))$, whenever $\min(r,c)>1$:
\begin{align}\nodes(\K(\rectangle_{r,c}))=\nodes(\K(\rectangle_{r-1,c}))+\nodes(\K(\rectangle_{r,c-1}))-\nodes(\K(\rectangle_{r-1,c-1}))+\binom{r+c-2}{r-1}. \label{eq:reccur}
\end{align}
By applying the induction hypothesis to \cref{eq:reccur} we obtain 

\begin{align*}\nodes(\K(\rectangle_{r,c}))&=\binom{r+c-1}{r-1}-1+\binom{r+c-1}{c-1}-1-\binom{r+c-2}{r-1}+1+\binom{r+c-2}{r-1} \\
&=\binom{r+c}{r}-1.\qedhere
\end{align*}
\end{proof}

\paragraph{Number of leaves in the game tree} 
The number of leaves in the game tree of a given position corresponds to the number of different games that can be played on that position. 
As expected, for both games this number can be exponential in $n$, for example for staircase partitions (see \cref{tab:tree-parm}).
Depending on the input partition, the number of leaves may be much smaller.
\cref{tab:tree-parm} also can be used to obtain the following result. 

\begin{lemma}\label{lem:minleaves}
Let $n$ be an integer and let $\lambda\vdash n$. 
The number of leaves in the game tree of $\K(\lambda)$ is at least $1$.
The number of leaves in the game tree of $\L(\lambda)$ is at least $n+1$.
Both lower bounds are tight.
\end{lemma}
 
\begin{proof}
Regarding the case of \King{}, it is enough to inspect the second row of \cref{tab:tree-parm}, where one may observe that this bound is achieved by %
$\br{n}$.
We now consider the case of LCTR.

We prove the claim by showing that among all of the partitions of $n>0$, the partition $\br{n}$ 
has the smallest number of leaves, in particular 
$\leaf(\L(\br{n}))=n+1$.
Consider now any partition $\lambda=\br{\lambda_1,\dots, \lambda_r}\vdash n$.
Let 
$\bar \lambda = \br{\lambda_1,\dots, \lambda_r-1}$ and let $\bar{\bar \lambda} = \br{\lambda_1+1, \ldots, \lambda_r-1}$ be a partition obtained from $\bar \lambda$ by adding one box to the first row.
 Now observe that since $\lambda\neq \br{}$
it follows that $\lambda[r-1,\lambda_r-1]=\br{1}$. 
The leaves of $\L(\bar \lambda)$ are the same as those of $\L(\lambda)$ except we remove both children (which are both leaves) of all the occurrences of $\L(\lambda)[r-1,\lambda_r-1]$ in the game tree, and all occurrences of $\L(\lambda)[r-1,\lambda_r-1]$ become leaves. 
Hence the number of leaves decreases by the number of occurrences of $\L(\lambda)[r-1,\lambda_r-1]$, in particular, the number of leaves decreases by at least one.
On the other hand notice that $\L(\bar{\bar \lambda})$ admits exactly one more leaf than $\L(\bar \lambda)$ since in the game tree of $\L(\bar \lambda)$ we replace the single occurrence of the leaf
$\lambda[0,\lambda_1]=\br{}$ with a node with two leaf children to get the game tree of $\L(\bar{\bar \lambda})$.
It follows that the number of leaves did not increase from $\lambda$ to $\bar{\bar \lambda}$ while $\lambda,\bar{\bar \lambda} \vdash n$.
Any partition $\lambda \vdash n $ minimizing the number of leaves can therefore be gradually transformed to $\br{n}$, without increasing the number of leaves.
It remains to observe, e.g. by \cref{tab:tree-parm,cor:leaf+node}, that $\br{n}$ indeed admits exactly $n+1$ leaves, as stated by the claim.
\end{proof}

The next lemma describes the number of leaves in the game tree of LCTR, when played on a rectangular partition. 

\begin{lemma}
For positive integers $r,c$ we have that  $\leaf(\L(\rectangle_{r,c}))=\binom{r+c}{r}$.
\end{lemma}
\begin{proof}
Similarly as above, $\min(r,c)=1$ implies $\rectangle_{r,c}=\gama_{r,c}$, so in those cases the claim can be easily verified (see also \cref{tab:tree-parm}).
For the remaining cases, by the principle of inclusion-exclusion we get
$$\leaf(\L(\rectangle_{r,c}))=\nodes(\K(\rectangle_{r,c}))-\nodes(\K(\rectangle_{r-1,c-1}))+\binom{r+c-2}{r-1}=\binom{r+c}{r},$$
where the first binomial symbol  corresponds to the fact that all $\binom{r+c-2}{r-1}$
occurrences of 
subposition $\L(\lambda)[r-1,c-1]$ give rise to two leaves.
\end{proof}
}{ Recall that $\nodes(A)$, and $\leaf(A)$, denote the number of nodes, and leaves, in the game tree of the game $A$, respectively.
The following corollary connects those notions, and follows immediately from \cref{def:truncation}.
\begin{corollary}\label{cor:leaf+node}
Let $\lambda$ be a partition. Then $\nodes(\L(\lambda))=\nodes(\K(\lambda))+\leaf(\L(\lambda))$.
\end{corollary}
Given a game $A$ of either type, recall that a position may occur several times within the game tree of $A$. 
More precisely, if $A[i,j]\neq \br{}$ then the pair $(i,j)$ gives rise to exactly $\binom{i+j}{i}$ nodes within the game tree of $A$.

The number of leaves in the game tree of a given position corresponds to the number of different games that can be played on that position. 
As expected, for both games this number can be exponential in $n$, for example for staircase partitions (see \cref{tab:tree-parm}).
Depending on the input partition, the number of leaves may be much smaller.
\cref{tab:tree-parm} also can be used to obtain the following result. 

\begin{proposition}\label{lem:minleaves}
The number of leaves in the game tree of $\K(\lambda)$ is at least $1$.
The number of leaves in the game tree of $\L(\lambda)$ is at least $n+1$.
Both lower bounds are tight.
\end{proposition}

}

\begin{table}[!ht]
    \centering
    \begin{tabular}{c|c|c|c}
         game $A$ & number of states & $\nodes(A)$  &  $\leaf(A)$ \\ \hline\hline
 $\L(\gama_{r,c}),\min(r,c)=1$ & $r+c$ & $2r+2c-1$ & $r+c$ \\
 $\K(\gama_{r,c}),\min(r,c)=1$  & $r+c-1$ & $r+c-1$ & $1$ \\
 $\L(\gama_{r,c}),\min(r,c)>1$ & $r+c-1$ & $2r+2c-1$ & $r+c$ \\
 $\K(\gama_{r,c}),\min(r,c)>1$ & $r+c-2$ & $r+c-1$ & $2$ \\ \hline
 $\L(\staircase_r)\vphantom{2^{2^{2^{2}}}}$ & $r+1$ & $2^{r+1}-1$ & $2^r$ \\
 $\K(\staircase_r)\vphantom{2^{2^{2}}}$ & $r$ & $2^{r}-1$ & $2^{r-1}$ \\ \hline
 $\L(\rectangle_{r,c})\vphantom{2^{2^{2^{2}}}}$ & $rc+1$  &  $2\binom{r+c}{r}-1$ & $\binom{r+c}{r}$ \\ 
 $\K(\rectangle_{r,c})\vphantom{2^{2^{2^{2}}}}$ & $rc$ & $\phantom{2}\binom{r+c}{r}-1$ & $\binom{r+c-2}{r-1}$ \\
    \end{tabular}
    \caption{Same games together with corresponding game tree parameters.
    }
    \label{tab:tree-parm}
\end{table}

For some simple families we present the corresponding values in \cref{tab:tree-parm}. 
The first two rows of \cref{tab:tree-parm} show the tightness of \cref{lem:minleaves}, i.e. they describe the asymptotically smallest number of nodes and leaves that the game tree of LCTR, or \King{}, may have. 
We believe that the number of leaves attained by the above lemma attains the maximal value whenever $r=c$, which leads to the following conjecture. 

\begin{conjecture}
The number of nodes in the game tree of any game of \King{} and LCTR on a partition of size $n^2$ is maximized in the partition $\rectangle_{n, n}$, i.e. the square partition.
    \end{conjecture}

\subsection{State-space complexity}\label{sec:statespace}
The state-space complexity of a game is the number of distinct game positions reachable from the initial position of the game. 
\begin{proposition}
Let $\lambda$ be a partition such that $\lambda\vdash n$.
Then the number of states of $\L(\lambda)$ and $\K(\lambda)$ is bounded above by $n+1$ and $n$, respectively.
On the other hand, the number of states of both $\L(\lambda)$ and $\K(\lambda)$ is bounded below by $\Omega(\sqrt{n})$.
\end{proposition}
\begin{proof}
Concerning both upper bounds, it is easy to notice that the number of positions is trivially bounded by the number of boxes of $\lambda$.
More precisely, the state space of any LCTR game may attain the bound of $n+1$, due to an additional position $\br{}$, while the state space of \King{} cannot exceed $n$. 
One can attain both upper bounds by considering, for instance, a gamma-shaped partition.

Regarding the lower bound, observe that no position in either game can be repeated, which implies that the number of states is at least one more than the height of the game tree.
Thus, 
in the case $\lambda=\staircase_r$ we have $n=\binom{r+1}{2}$ and 
$\K(\lambda)$ admits $r=\sqrt{1/4+2n}-1/2$ states, while
$\L(\lambda)$ admits $r+1$ states, cf. \cref{tab:tree-parm}, which matches the length of any game played on $\staircase_r$. 

Notice that any partition $\lambda \vdash n\geq {r+1 \choose 2}$ with $\lambda\neq \staircase_r$ has a game tree of height at least $r$ in \King{} and $r+1$ in LCTR.
\end{proof}

\section{Open Problems}
\label{sec:open}

In this section, we give some possible directions for future research.

\paragraph{Alternative representations of input partition:}
There are various ways to represent the input partition in a possibly more compact way. 
    For instance, one may use exponents to represent rows of same length (see \cref{sec:partitions}).
    A similar approach is by only storing two-dimentional positions of all corners of the given partition. 
    For example, the exponential and corner representations of $\br{6,5,4,4,4,2,2}$ are, respectively, 
    \[
    \br{6,5,4^3,2^2} \text{\quad and }\quad \{(6,1),(5,2),(4,5), (2,7)\}.
    \]
    
    If we assume exponential or corner representation of input partition, the computational difficulty of the problem may change nontrivially, as converting between the mentioned representations may take as much as $O(n)$ time units.
    
    On one hand, doing this can  speed up certain parts of computation of $\G$ values, while other difficulties may appear. 
    For instance, it is not clear to us whether one can always determine the Durfee length in $O(\log n)$ time units using exponential notation. 
    It would be interesting to see if the overall time complexity can be improved by assuming that the input partition is encoded using exponential notation. 

\paragraph{Multidimensional LCTR:}
Let us denote the tuples $v=(v_1,...,v_d)$ with $v_i\ge 1$ for all $i$ as \emph{points}.
    For two such points $v,v'$, we say that $v\le v'$ if $v_i\le v'_i$ for all $i$. 
    The position of $d$-dimensional LCTR is simply a finite set $S$ of non-comparable points, also called \emph{corners}.
    In $d$-dimensional LCTR on their turn a player chooses $i\in \{1,...,d\}$, and 
    modifies every $s\in S$ by decreasing $s_i$ by one
    if $s_i-1$ is positive, and removing $s$ from $S$ otherwise. 
    The game ends when $S$ is empty.
    
    In such setting it is easy to see that the Sprague-Grundy values may not exceed $d$, however zero values do not propagate, hence there is no hope for the similar algorithm to work.
    Is there a fast algorithm like in the case $d=2$? 
    
\paragraph{Impartial chess on an arbitrary partition}
    \emph{Impartial} chess is an impartial game played on a rectangular partition $\rectangle_{r,c}$ with a single chess-piece starting on the top-left most  position $\rectangle_{r,c}[0,0]$.
    The move from a position $\rectangle_{r,c}[i,j]$ to a position $\rectangle_{r,c}[i',j']$ is allowed if (i) it respects the movement rules of the regular game of chess, and (ii) $i\leq i'< r$, as well as $j\leq j'< c$. 
    
    The game of \King{} presented in this paper is very similar to  Impartial chess with a king.
    The two differences are: (i) the king cannot move diagonally, and (ii) the initial partition in \King{} need not be of rectangular shape. 
    Despite the similarity, the Impartial chess with a king seems to be much harder to understand when played on an arbitrary partition.
    On the other hand, Impartial chess with a rook is equivalent to two-pile Nim, while Impartial chess with a queen is equivalent to the Wythoff game \cite{wythoff1907modification}.
    
    Due to the above connections of Impartial chess with other impartial games, it would be interesting to study this game in the generalized setting where the starting position can be any partition.
    
\paragraph{Non-repetitive variant:}
The majority of the winning strategy of both LCTR and \King{} consist of mirroring the opponent's moves.
Suppose that we require that within every four consecutive moves, the number of $L$ moves is distinct from the number of $T$ moves. Does the game becomes significantly harder? 

\subsection*{Acknowledgements}

We are grateful to the anonymous referees whose suggestions improved the quality of our paper. 
Specifically, 
\cref{sec:LCTR misere} was enhanced while 
\cref{sec:hierarchy} was added
in light of \cite{GURVICH201854}. 

This work is supported in part by
an internal grant from Rhodes College as well as by the Slovenian Research And Innovation Agency (research program P1-0383, research projects N1-0160, J1-3003, J1-4008 and J5-4423, and the bilateral research project BI-US/22-24-164).

\bibliography{bibl}
\bibliographystyle{elsarticle-num}

\end{document}